\RequirePackage{fix-cm}

\documentclass[]{article}  

\usepackage{amsmath}
\usepackage{agor}
\usepackage{bbm}
\usepackage[mathlines,pagewise,left]{lineno}
\usepackage{amssymb}
\usepackage{color}
\usepackage{multirow}
\usepackage{caption}
\usepackage{comment}
\usepackage{algorithm,algorithmicx}
\usepackage{algpseudocode}
\usepackage{graphicx} 
\usepackage{booktabs} 
\usepackage{threeparttable}
\usepackage{multirow}
\usepackage{url}
\usepackage{bm}
\usepackage{pdflscape}
\usepackage[export]{adjustbox}

\newcommand{\Gsp}{\ensuremath{\mathcal{G}_{SP}}}
\newcommand{\Gmrf}{\ensuremath{\mathcal{G}_{MRF}}}

\newcommand{\ignore}[1]{}

\begin{document}

\title{A Graph-based Decomposition Method for Convex Quadratic Optimization with Indicators \thanks{This research is supported, in part, by  NSF grants  2006762 and 2007814.}
}


\author{Peijing Liu\thanks{Daniel J. Epstein Department of Industrial and Systems Engineering,
		University of Southern California,
		Los Angeles, CA, USA. 
		\texttt{peijingl@usc.edu} }         \and
        Salar Fattahi\thanks{Industrial and Operations Engineering, University of Michigan, Ann Arbor, MI, USA. \texttt{fattahi@umich.edu} } \and Andr\'es G\'omez\thanks{Daniel J. Epstein Department of Industrial and Systems Engineering,
        	University of Southern California,
        	Los Angeles, CA, USA. 
        	\texttt{gomezand@usc.edu}} \and
        Simge K{\"u}{\c{c}}{\"u}kyavuz\thanks{Department of Industrial Engineering and Management Sciences,
        	Northwestern University,
        	Evanston, IL, USA. 	
        	\texttt{simge@northwestern.edu} } 
}


\maketitle

\begin{abstract}
In this paper, we consider convex quadratic optimization problems with indicator variables when the matrix $Q$  defining the quadratic term in the objective is sparse. We use a graphical representation of the support of $Q$, and show that if this graph is a path, then we can solve the associated problem in polynomial time. This enables us to construct a 
 compact extended formulation for the closure of the convex hull of the epigraph of the mixed-integer convex problem. 
 Furthermore,  we propose a novel decomposition method for general (sparse) $Q$, which leverages the efficient algorithm for the path case. Our  computational experiments demonstrate  the effectiveness of the proposed method compared to state-of-the-art  mixed-integer optimization solvers.  
\end{abstract}



\section{Introduction}\label{sec1}
Given a positive semi-definite matrix $Q\in \R^{n\times n}$ and vectors $a,c\in \R^n$, we study the mixed-integer quadratic optimization problem
\begin{subequations}\label{eq:miqo}
\begin{align}
    \min_{x\in \R^n,z\in \{0,1\}^n}\;&a^\top z+c^\top x+\frac{1}{2}x^\top Qx\label{eq:miqo_obj}\\
    \text{s.t.}\;& x_i(1-z_i)=0&i=1,\dots,n.\label{eq:miqo_complementary}
\end{align}
\end{subequations}
Binary vector of indicator variables, $z$, is used to model the support of the vector of continuous variables, $x$. Indeed, if $a_i> 0$, then $z_i=1\Leftrightarrow x_i\neq 0$. Problem \eqref{eq:miqo} arises in portfolio optimization \cite{B:miqp}, sparse regression problems \cite{Bertsimas2016,cozad2014learning}, and probabilistic graphical models \cite{MKS21,KSMW20}, among others.

\subsection{Motivation: Inference with graphical models} 

A particularly relevant application of Problem \eqref{eq:miqo} is in sparse 
inference problems with Gaussian Markov random fields (GMRFs). Specifically, we consider
a special class of GMRF models known as Besag models \cite{besag1974spatial}, which are widely 
used in the literature \cite{besag1991bayesian,besag1995conditional,geman1986markov,hochbaum2001efficient,saquib1998ml,wu2010maximum} to 
model spatio-temporal processes including image restoration and computer vision, disease
mapping, and evolution of financial instruments. Given an undirected graph 
$\Gmrf=(N,E)$ with vertex set $N$ and edge set $E$, where edges encode adjacency
relationships, and given distances $d_{ij}$ associated with each edge, consider a multivariate random variable $V\in \mathbb{R}^N$  indexed by 
the vertices of $\Gmrf$ with probability distribution
\[
p(V) \, \propto \, \exp\left({-}\!\!\!\sum_{(i,j)\in E}\frac{1}{d_{ij}}(V_i-
V_j)^2\right),
\]
This probability distribution encodes the prior belief that adjacent variables have 
similar values. The values of $V$ cannot be observed directly, but rather some noisy 
observations $y$ of $V$ are available, where $y_i=V_i+\varepsilon_i$, with 
$\varepsilon_i\sim\mathcal{N}(0,\sigma_i^2)$.  Figure~\ref{fig: 2D} depicts a sample 
GMRF commonly used to model spatial processes, where edges correspond to horizontal 
and vertical adjacency. 

\begin{figure}[!h]
	\begin{center}
		\includegraphics[width=0.45\linewidth, trim={4cm 4cm 12cm 1cm},clip]{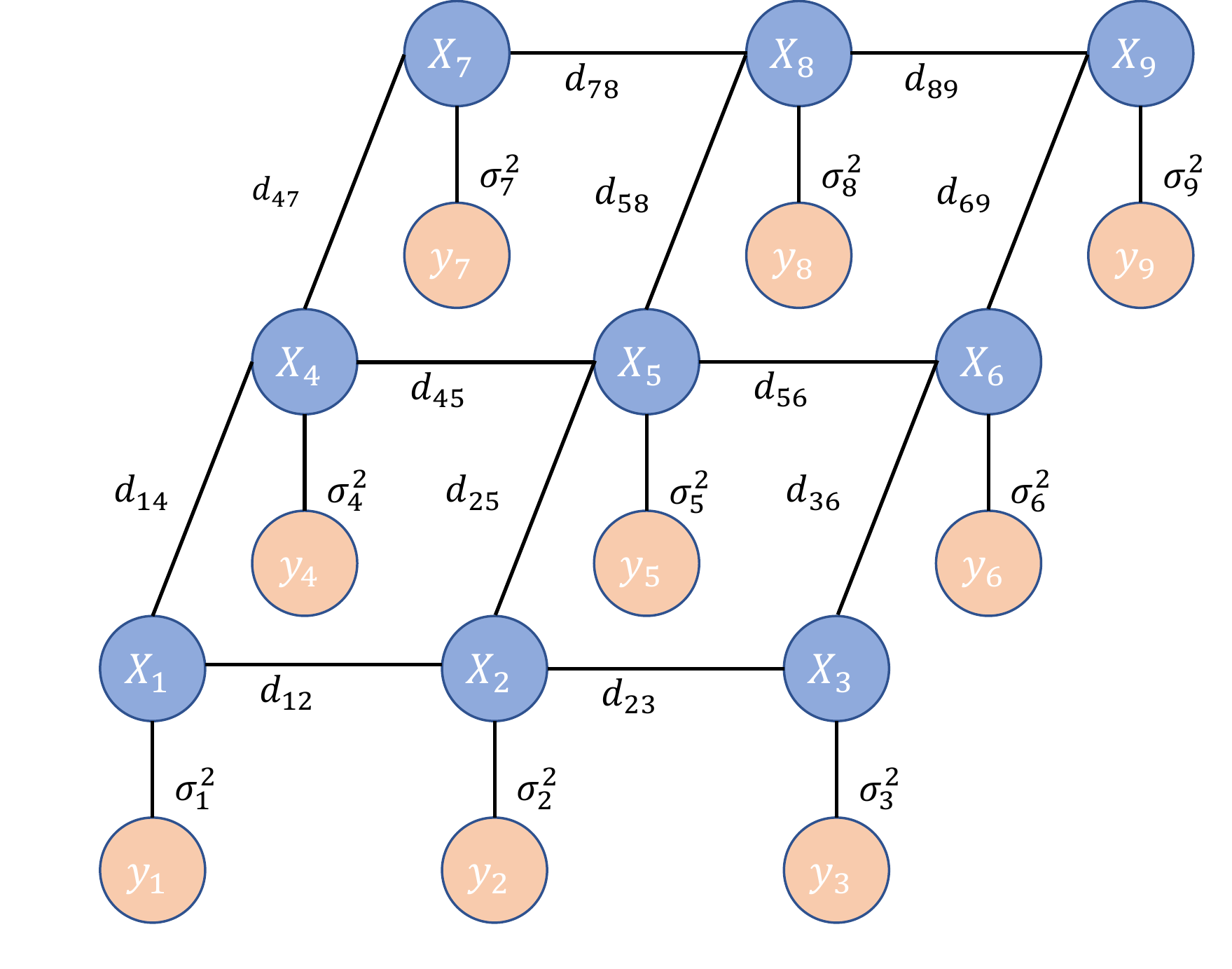}
		\caption{Two-dimensional GMRF.}	
		\label{fig: 2D}
	\end{center}
\end{figure}

In this case, the maximum a posteriori estimate of the true values of $V$ 
can be found by solving the optimization problem
\begin{equation}\label{eq:map}
\min_x\;\sum_{i\in N}\frac{1}{\sigma_i^2}(y_i-x_i)^2+\sum_{(i,j)\in E}\frac{1}{d_{ij}}(x_i-x_j)^2.
\end{equation}
Problem \eqref{eq:map} can be solved in closed form when there are no additional restrictions on the random variable. However, we consider the
situation where the random variable is also assumed to be sparse 
\cite{atamturk2021sparse}.  For example, few pixels in an image may be salient from the
background, few geographic locations may be affected by an epidemic, or the underlying 
value of a financial instrument may change sparingly over time.  Moreover, models such as
\eqref{eq:map} with sparsity have also been proposed to estimate precision matrices of 
time-varying Gaussian processes~\cite{fattahi2021scalable}. In all cases, the sparsity prior can 
be included in model \eqref{eq:map} with the inclusion of the $\ell_0$ term 
$\sum_{i\in N} a_iz_i$, where $a$ is a penalty vector and binary variable $z_i$ indicates whether the corresponding continuous variable $x_i$ is nonzero, for $i\in N$ . This results in an optimization problem of the form \eqref{eq:miqo}:
\begin{subequations}\label{eq:map_miqo}
\begin{align}
\min_{x,z}\;&\sum_{i\in N}\frac{1}{\sigma_i^2}(y_i-x_i)^2+\sum_{(i,j)\in E}\frac{1}{d_{ij}}(x_i-x_j)^2+\sum_{i\in N} a_iz_i\label{eq:map_miqo_obj}\\
\text{s.t.}\;&-Mz_i\leq x_i\leq Mz_i\qquad \forall i\in N\label{eq:map_miqo_M}\\
&x\in \R^N,\; z\in \{0,1\}^N.
\end{align}
\end{subequations}
Note that constraint \eqref{eq:map_miqo_M} corresponds to the popular big-M linearization of the complementarity constraints \eqref{eq:miqo_complementary}. In this case, it can be shown that setting $M=\max_{i\in N}y_i-\min_{i\in N}y_i$ results in a valid mixed-integer optimization formulation. Therefore, it is safe to assume that $(x,z)$ belongs to a compact set $\mathcal{X}=\{(x,z)\in \R^N\times [0,1]^N: -Mz\leq x\leq Mz\}$.

\subsection{Background} \label{sec:background}
Despite problem \eqref{eq:miqo} being NP-hard \cite{chen2017strong}, there has been tremendous progress towards solving it to optimality. Due to its worst case complexity, a common theme for successfully solving \eqref{eq:miqo} is the development of theory and methods for special cases of the problem, where matrix $Q$ is assumed to have a special structure, providing insights for the general case. For example, if matrix $Q$ is diagonal (resulting in a fully separable problem), then problem \eqref{eq:miqo} can be cast as a convex optimization problem via the perspective reformulation \cite{Ceria1999}. This convex hull characterization has led to the development of several techniques for problem \eqref{eq:miqo} with general $Q$, including cutting plane methods \cite{Frangioni2006,frangioni2017improving}, strong MISOCP formuations \cite{akturk2009strong,Gunluk2010}, approximation algorithms \cite{xie2020scalable}, specialized branching methods \cite{hazimeh2020sparse}, and presolving methods \cite{atamturk2020safe}. Recently, problem \eqref{eq:miqo} has been studied under other structural assumptions, including: quadratic terms involving two variables only \cite{anstreicher2021quadratic,atamturk2018strong,frangioni2018decompositions,hga:2x2,Jeon2017}, rank-one quadratic terms \cite{atamturk2019rank,atamturk2020supermodularity,wei2020convexification,wei2020ideal}, and quadratic terms with Stieltjes matrices \cite{atamturk2021sparse}. If the matrix can be factorized as $Q=Q_0^\top Q_0$ where $Q_0$ is sparse (but $Q$ is dense), then problem \eqref{eq:miqo} can be solved (under appropriate conditions) in polynomial time \cite{del2020subset}. Finally, in \cite{das2008algorithms}, the authors show that if the sparsity pattern of $Q$ corresponds to a tree with maximum degree $d$, and all coefficients $a_i$ are identical, then a cardinality-constrained version of problem \eqref{eq:miqo} can be solved in $\mathcal{O}(n^3d)$ time---immediately leading to an $\mathcal{O}(n^4d)$ algorithm for the regularized version considered in this paper.

We focus on the case where matrix $Q$ is sparse, and explore efficient methods to solve  problem \eqref{eq:miqo}. Our analysis is closely related to the support graph of $Q$, defined below.
\begin{definition}
Given matrix $Q\in \R^{n\times n}$, the support graph of $Q$ is an undirected graph $\mathcal{G}=(N,E)$, where $N=\{1,\dots,n\}$ and, for $i<j$, $(i,j)\in E\Leftrightarrow Q_{ij}\neq 0$.
\end{definition}
 Note that we may assume without loss of generality that graph $\mathcal{G}$ is connected, because otherwise problem \eqref{eq:miqo} decomposes into independent subproblems, one for each connected component of $\mathcal{G}$. 

\subsection{Contributions and outline} 
In this paper, we propose new algorithms and convexifications for problem \eqref{eq:miqo} when $Q$ is sparse. First, in Section~\ref{sec:path}, we focus on the case when $\mathcal{G}$ is a path. We propose an $\mathcal{O}(n^2)$ algorithm for this case, which improves upon the complexity resulting from the algorithm in \cite{das2008algorithms} without requiring any assumption on vector $a$. Moreover, we provide a compact extended formulation for the closure of the convex hull of $$X=\left\{(x,z,t)\in \R^n\times\{0,1\}^n\times \R: t\geq x^\top Qx,\;x_i(1-z_i)=0, \;\forall i\in N \right\}$$  for cases where  $\mathcal{G}$ is a path, requiring $\mathcal{O}(n^2)$ additional variables. In Section~\ref{sec:fenchel}, we propose a new method for general (sparse) $Q$, which leverages the efficient algorithm for the path case. In particular, using Fenchel duality, we relax selected quadratic terms in the objective \eqref{eq:miqo_obj}, ensuring that the resulting quadratic matrix has a favorable structure. In Section~\ref{sec:decomposition}, we elaborate on how to select the quadratic terms to relax. Finally, in Section~\ref{sec:computations}, we present computational results illustrating that the proposed method can significantly outperform off-the-shelf mixed-integer optimization solvers.  

\subsection{Notation} 
Given a matrix $Q\in \R^{n\times n}$ and indices $0\leq i<j\leq n+1$, we denote by $Q[i,j]\in\mathbb{R}^{(j-i-1)\times (j-i-1)}$ the submatrix of $Q$ from indices $i+1$ to $j-1$. Similarly, given any vector $c\in \R^n$, we denote by $c[i,j]\in\mathbb{R}^{j-i-1}$ the subvector of a vector $c$ from indices $i+1$ to $j-1$. Given a set $S\subseteq \R^{n}$, we denote by $\conv(S)$ its convex hull and by $\text{cl}\ \conv(S)$ the closure of its convex hull.

\section{Path Graphs}\label{sec:path}

In this section, we focus on the case where graph $\mathcal{G}$ is a path, that is, there exists a permutation function $\pi:\{1,\dots,n\}\to \{1,\dots,n\}$ such that $(i,j)\in E$ if and only if $i=\pi(k)$ and $j=\pi(k+1)$ for some $k=1,2,\dots, n-1$.
Without loss of generality, we assume variables are indexed such that $\pi(k)=k$, in which case matrix $Q$ is tridiagonal and problem \eqref{eq:miqo} reduces to
\begin{subequations}\label{eq:tridiagonalMIQO}
\begin{align}
\zeta=\min_{x\in \R^n,z\in \{0,1\}^n}\;&a^\top z+c^\top x+\frac{1}{2}\sum_{i=1}^{n}Q_{ii}{x_i}^2+\sum_{i=1}^{n-1}{Q_{i,i+1}}x_ix_{i+1}\\
    \text{s.t.}\;& x_i(1-z_i)=0\qquad\qquad\qquad i=1,\dots,n.
\end{align}
\end{subequations}

Problem~\eqref{eq:tridiagonalMIQO} is interesting in its own right: it has immediate applications  in the estimation of one-dimensional graphical models~\cite{fattahi2021scalable} (such as time-varying signals), as well as sparse and smooth signal recovery~\cite{mao2018spatio,atamturk2021sparse,ziniel2010tracking}. In particular, suppose that our goal is to estimate a sparse and smoothly-changing signal $\{x_t\}_{t=1}^n$ from observational data $\{y_t\}_{t=1}^n$. This problem can be written as the following optimization:
\begin{subequations}\label{eq:time-series}
\begin{align}
\min_{x\in \R^n,z\in \{0,1\}^n}\;&a^\top z+\sum_{t=1}^n(x_t-y_t)^2 + \sum_{t=1}^{n-1}(x_{t+1}-x_{t})^2\\
    \text{s.t.}\;& x_t(1-z_t)=0&\hspace{-4cm}t=1,\dots,n.
\end{align}
\end{subequations}
The first term in the objective promotes sparsity in the estimated signal, while the second and third terms promote the closeness of the estimated signal to the observational data and its temporal smoothness, respectively. It is easy to see that~\eqref{eq:time-series} can be written as a special case of~\eqref{eq:tridiagonalMIQO}.

First, we discuss how to solve \eqref{eq:tridiagonalMIQO} efficiently as a shortest path problem. For simplicity, we assume that $Q\succ 0$ (unless stated otherwise).

\subsection{A shortest path formulation}\label{sec:shortestPath}
In this section, we explain how to solve~\eqref{eq:tridiagonalMIQO} by solving a shortest path problem on an auxiliary directed acyclic graph (DAG). Define for $0\leq i<j\leq n+1$
{\small
\begin{align}
    w_{ij}&\defeq\sum_{k=i+1}^{j-1}a_k+\min_{x[i,j]\in\mathbb{R}^{j-i-1}}\;\left\{ \sum_{k=i+1}^{j-1}c_kx_k+\frac{1}{2}\sum_{k=i+1}^{j-1}Q_{kk}{x_k}^2+\sum_{k=i+1}^{j-2}{Q_{k,k+1}}x_kx_{k+1}\right\}\notag\\
    &=\sum_{k=i+1}^{j-1}a_k-\frac{1}{2}c[i,j]^\top Q[i,j]^{-1}c[i,j],\label{eq:definition_w_ij}
\end{align}
}
where the equality follows from the fact that \begin{equation}\label{eq:defXstar}x^*(i,j)=-Q[i,j]^{-1}c[i,j]\end{equation} is the corresponding optimal solution. By convention, we let 
$w_{i,i+1}=0$ for all $i=0,\ldots,n$.

We start by discussing how to solve a restriction of problem \eqref{eq:tridiagonalMIQO} involving only continuous variables. Given any fixed $\bar z\in \{0,1\}$, let $x(\bar z)$ be the unique minimizer of the optimization problem
\begin{subequations}\label{eq:tridiagonalMIQOFixedZ}
\begin{align}
\zeta(\bar z)=a^\top \bar z+\min_{x\in \R^n}\;&\left\{c^\top x+\frac{1}{2}\sum_{i=1}^{n}Q_{ii}{x_i}^2+\sum_{i=1}^{n-1}{Q_{i,i+1}}x_ix_{i+1}\right\}\\
    \text{s.t.}\;& x_i(1-\bar z_i)=0&\hspace{-5cm}i=1,\dots,n.
\end{align}
\end{subequations}
Lemma~\ref{lem:fixedZ} discusses the structure of the optimal solution $x(\bar z)$ in \eqref{eq:tridiagonalMIQOFixedZ}, which can be expressed using the optimal solutions $x^*(i,j)$ of subproblems given in \eqref{eq:defXstar}.
\begin{lemma}\label{lem:fixedZ}
Let $0=v_0< v_1<v_2<\dots<v_\ell<v_{\ell+1}= n+1$ be the indices such that $\bar z_{j}=0$ if and only if  $j=v_k$ for some $1\le k \le \ell$ and $\ell\in \{0,\dots,n\}$. Then $x(\bar z)_{v_k}=0$ for $k=1,\dots,\ell$, and $x(\bar z)[v_k,v_{k+1}]=x^*(v_k,v_{k+1})$. Finally, the optimal objective value is $\zeta(\bar z)=\sum_{k=0}^\ell w_{v_k,v_{k+1}}$.
\end{lemma}
\begin{proof}
Constraints $x_{v_k}(1-\bar z_{v_k})=0$ and $\bar z_{v_k}=0$ imply that $x_{v_k}=0$ in any feasible solution. Moreover, note that since $x_{v_k}=0$ for all $k=1\dots,\ell$, problem \eqref{eq:tridiagonalMIQOFixedZ} decomposes into $\ell+1$ independent subproblems, each involving variables $x[v_k,v_{k+1}]$ for $k=0,\dots,\ell$. Note that some problems may contain no variables and are thus trivial. Finally, by definition, the optimal solution of those subproblems is precisely $x^*(v_k,v_{k+1})$. The optimal objective value can be verified simply by substituting $x$ with its optimal value. 
\end{proof}

Lemma~\ref{lem:fixedZ} shows that, given the optimal values for the indicator variables,  problem \eqref{eq:tridiagonalMIQO} is decomposable into smaller subproblems, each with a closed-form solution. This key property suggests that \eqref{eq:tridiagonalMIQO} can be cast as a shortest path (SP) problem.

\begin{definition}[SP graph]\label{def:SPgraph}
Define the weighted directed acyclic graph \Gsp\ with vertex set $N\cup\{0,n+1\}$, arc set $A=\{(i,j)\in \mathbb{Z}_+^2: 0\leq i<j\leq n+1\}$ and weights $w$ given in \eqref{eq:definition_w_ij}. Figure~\ref{fig:SPgraph} depicts a graphical representation of \Gsp.
\end{definition}

\begin{figure}[h!]
\centering
\includegraphics[trim={0cm 4.5cm 1cm 0cm}, clip, width=0.9\textwidth]{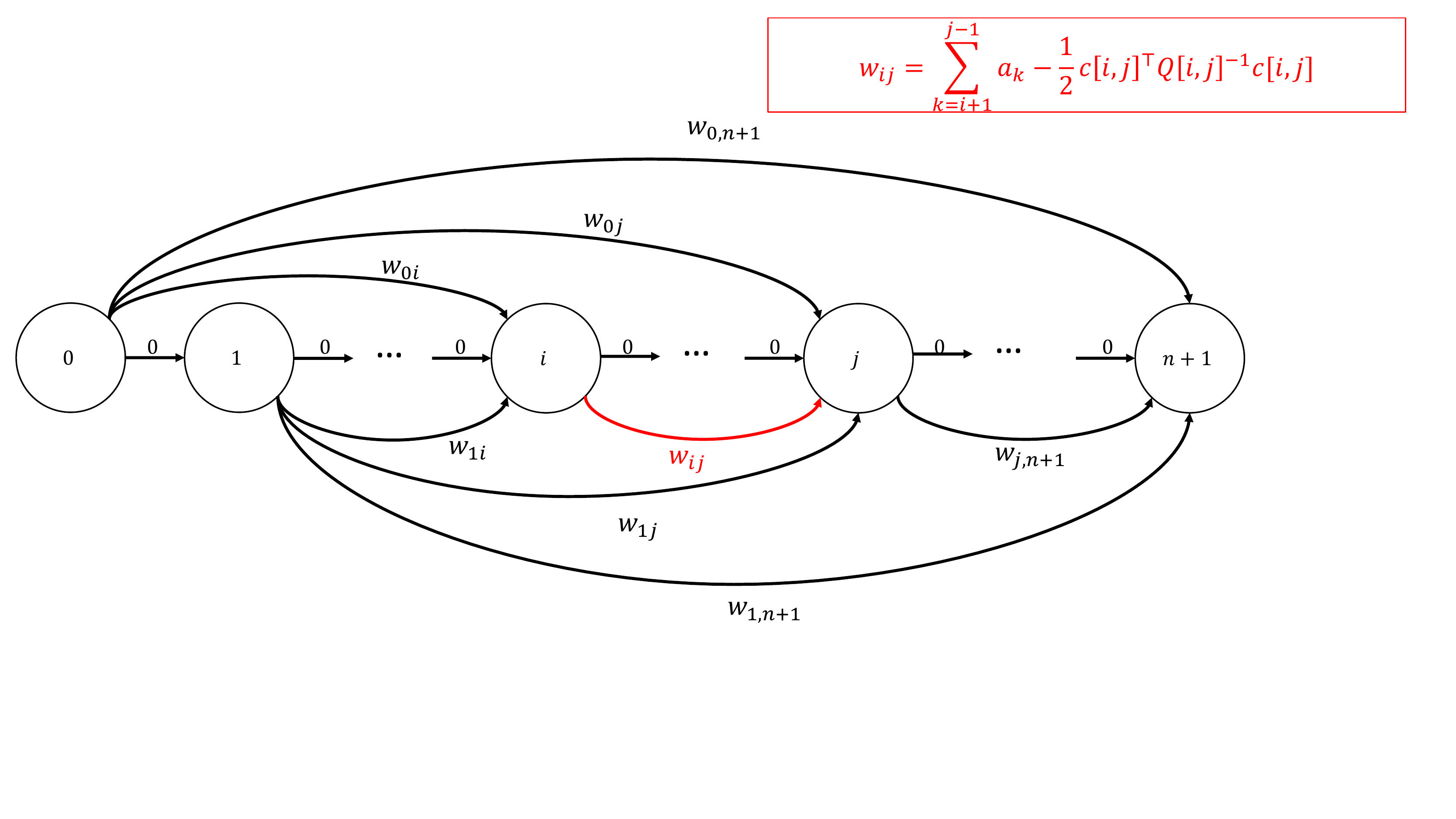}
\caption{Graphical depiction of \Gsp.}
\label{fig:SPgraph}
\end{figure}

\begin{proposition}\label{prop:tridiagoanl_equal_SP}
The length of any $(0,n+1)$-path $p=\{0,v_1,\ldots,v_\ell,n+1\}$ on \Gsp\ is the objective value of the solution of problem  \eqref{eq:tridiagonalMIQO} corresponding to setting $\bar z_v=0$ if and only if $v\in p$, and setting $x=x(\bar z)$.
\end{proposition}
\begin{proof}
There is a one-to-one correspondence between any path $p$ on \Gsp, and the solution $(x(\bar z),\bar z)$ where $\bar z$ is given as in Lemma~\ref{lem:fixedZ}, that is, $\bar z_{j}=0\Leftrightarrow j=v_k$ for some $k=1,\dots,\ell$. By construction, the length of the path is precisely the objective value associated with $(x(\bar z),\bar z)$. 
\end{proof}

Proposition~\ref{prop:tridiagoanl_equal_SP} immediately implies that the solution with smallest cost corresponds to a shortest path, which we state next as a corollary.

\begin{corollary}\label{cor:shortestPath}
An optimal solution $(x(z^*),z^*)$ of \eqref{eq:tridiagonalMIQO} can be found by computing a $(0,n+1)$-shortest path on \Gsp. Moreover, the solution found satisfies $z_i^*=0$ if and only if vertex $i$ is visited by the shortest path.
\end{corollary}

\subsection{Algorithm}\label{sec:efficient_alg}

Observe that since graph \Gsp\ is acyclic, a shortest path can be directly computed by a labeling algorithm in complexity linear in the number of arcs $|A|$, which in this case is $\mathcal{O}(n^2)$. Moreover, computing the cost $w_{ij}$ of each arc requires solving the system of equalities $Q[i,j]x[i,j]=-c[i,j]$, which can be done in $\mathcal{O}(n)$ time using Thomas algorithm~\cite[Chapter 9.4]{datta2010numerical}. Thus, the overall complexity of this direct method is $\mathcal{O}(n^3)$ time, and it requires $\mathcal{O}(n^2)$ memory to store graph \Gsp. We now show that this complexity can in fact be improved.

\begin{algorithm}[h]
	\caption{Algorithm for problem \eqref{eq:tridiagonalMIQO}}
	\label{alg:SP} \small
	\begin{algorithmic}[1]
		\renewcommand{\algorithmicrequire}{\textbf{Input:}}
		\renewcommand{\algorithmicensure}{\textbf{Output:}}
		\Require $a,c\in \R^n$, $Q\in R^{n\times n}$ tridiagonal positive definite.
		\Ensure Optimal objective value $\zeta$ of \eqref{eq:tridiagonalMIQO}.
		
		\State $\ell_0\leftarrow 0$
		\State $\ell_k\leftarrow \infty$ for $k=1,\dots,n+1$ \Comment{Shortest path labels, initially $\infty$}
		
		\For{$i= 0, \ldots, n$}\label{line:outer}
		\State $\ell_{i+1}\leftarrow \min\{\ell_{i+1},\ell_i\}$ \Comment{$w_{i,i+1}=0$}
		\State $\bar c\leftarrow 0$, $\bar q\leftarrow \infty$ \Comment{Stores linear and quadratic coefficients}
		\State $\bar w\leftarrow 0$ 
		\For{$j=i+2, \ldots, n+1$}\label{line:inner}
		\State $\bar c \leftarrow c_{j-1}-\frac{Q_{j-2,j-1}}{\bar q}\bar c$ \Comment{Assume $Q_{0,1}=0$}
		\State $\bar q \leftarrow Q_{j-1,j-1}-\frac{\left(Q_{j-2,j-1}\right)^2}{\bar q}$ \Comment{Assume $Q_{0,1}=0$}
		\State $\bar w\leftarrow \bar w -\frac{1}{2}\frac{\bar c^2}{\bar q}+a_{j-1}$ \Comment{$\bar w=w_{ij}$}\label{line:barK}
		\State $\ell_{j}=\min\{\ell_j, \ell_i+\bar w\}$\label{line:labelUpdate}
		\EndFor
		\EndFor  \label{line:endFor}
		\State \Return $\ell_{n+1}$
		
	\end{algorithmic}
\end{algorithm}

\begin{proposition}\label{prop:complexity}
Algorithm~\ref{alg:SP} solves problem \eqref{eq:tridiagonalMIQO} in $\mathcal{O}(n^2)$ time and using $\mathcal{O}(n)$ memory. 
\end{proposition}

Observe that since Algorithm~\ref{alg:SP} has two nested loops (lines~\ref{line:outer} and \ref{line:inner}) and each operation inside the loop can be done in $\mathcal{O}(1)$, the stated time complexity of $\mathcal{O}(n^2)$ follows. Moreover, Algorithm~\ref{alg:SP} only uses variables $\bar c, \bar q,\bar k,\ell_0,\dots,\ell_{n+1}$, thus the stated memory complexity of $\mathcal{O}(n)$ follows. Therefore, to prove Proposition~\ref{prop:complexity}, it suffices to show that Algorithm~\ref{alg:SP} indeed solves problem \eqref{eq:tridiagonalMIQO}. 

Algorithm~\ref{alg:SP} is based on the forward elimination of variables. Consider the optimization problem \eqref{eq:definition_w_ij}, which we repeat for convenience 
\begin{align}\label{eq:repeated}
w_{ij}&=\sum_{k=i+1}^{j-1}a_k+\min_{x[i,j]\in\mathbb{R}^{j-i-1}}\;\Big\{\sum_{k=i+1}^{j-1} c_kx_k\notag\\
&\qquad+\frac{1}{2}\sum_{k=i+1}^{j-1}Q_{kk}{x_k}^2+\sum_{k=i+1}^{j-2}{Q_{k,k+1}}x_kx_{k+1}\Big\}.
\end{align}
Lemma~\ref{lem:elimination} shows how we can eliminate the first variable, that is, variable $x_{i+1}$,  in \eqref{eq:repeated}.

\begin{lemma}\label{lem:elimination}
	If $j=i+2$, then $w_{ij}=a_{i+1}-\frac{c_{i+1}^2}{2Q_{i+1,i+1}}$. Otherwise, 
	\begin{align*}w_{ij}=&\;a_{i+1}-\frac{c_{i+1}^2}{2Q_{i+1,i+1}}+\sum_{k=i+2}^{j-1}a_k\\
	&+\min_{x[i+1,j]\in\mathbb{R}^{j-i-2}}\;\left\{\sum_{k=i+2}^{j-1} \tilde c_kx_k+\frac{1}{2}\sum_{k=i+2}^{j-1}\tilde Q_{kk}{x_k}^2+\sum_{k=i+2}^{j-2}{Q_{k,k+1}}x_kx_{k+1}\right\},\end{align*}
	where $\tilde c_{i+2}=c_{i+2}-c_{i+1}\frac{Q_{i+1,i+2}}{Q_{i+1,i+1}}$, $\tilde c_k=c_k$ for $k>i+2$, $ \tilde Q_{i+2,i+2}=Q_{i+2,i+2}-\frac{Q_{i+1,i+2}^2}{Q_{i+1,i+1}}$, and $\tilde Q_{kk}=Q_{kk}$ for $k>i+2$. 
\end{lemma}
\begin{proof}
	If $j=i+2$ then the optimal solution of $\min_{x_{i+1}\in\mathbb{R}} \{a_{i+1}+c_{i+1}x_{i+1}+\frac{1}{2}Q_{i+1,i+1}x_{i+1}^2\}$ is given by $x_{i+1}^*=-c_{i+1}/Q_{i+1,i+1}$, with objective value $a_{i+1}-\frac{c_{i+1}^2}{2Q_{i+1,i+1}}$. Otherwise, from the KKT conditions corresponding to $x_{i+1}$, we find that 
	$$c_{i+1}+Q_{i+1,i+1}x_{i+1}+Q_{i+1,i+2}x_{i+2}=0\implies x_{i+1}=\frac{-c_{i+1}-Q_{i+1,i+2}x_{i+2}}{Q_{i+1,i+1}}.$$
	Substituting out $x_{i+1}$ in the objective value, we obtain the equivalent form 
	{\small
	\begin{align*}
	&\sum_{k=i+1}^{j-1}a_k-\frac{c_{i+1}^2}{Q_{i+1,i+1}}-\frac{c_{i+1}Q_{i+1,i+2}x_{i+2}}{Q_{i+1,i+1}}+\sum_{k=i+2}^{j-1} c_kx_k+\frac{1}{2}\frac{\left(-c_{i+1}-Q_{i+1,i+2}x_{i+2}\right)^2}{Q_{i+1,i+1}}\\&\qquad+\frac{1}{2}\sum_{k=i+2}^{j-1}Q_{kk}x_k^2
	-\frac{Q_{i+1,i+2}}{Q_{i+1,i+1}}c_{i+1}x_{i+2}-\frac{\left(Q_{i+1,i+2}x_{i+2}\right)^2}{Q_{i+1,i+1}}+\sum_{k=i+2}^{j-2}Q_{k,k+1}x_kx_{k+1}\\
	=&\sum_{k=i+1}^{j-1}a_k-\frac{c_{i+1}^2}{2Q_{i+1,i+1}}+\left(c_{i+2}-c_{i+1}\frac{Q_{i+1,i+2}}{Q_{i+1,i+1}}\right)x_{i+2}+\sum_{k=i+3}^{j-1}c_kx_k\\
	&\qquad+\frac{1}{2}\left(Q_{i+2,i+2}-\frac{Q_{i+1,i+2}^2}{Q_{i+1,i+1}}\right)x_{i+2}^2+\sum_{k=i+3}^nQ_{kk}x_k^2+\sum_{k=i+2}^{j-2}Q_{k,k+1}x_kx_{k+1}. 
	\end{align*} 

	} 
\end{proof}
The critical observation from Lemma~\ref{lem:elimination} is that, after elimination of the first variable, only the linear coefficient $c_{i+2}$ and diagonal term $Q_{i+2,i+2}$ need to be updated. From Lemma~\ref{lem:elimination}, we can deduce the correctness of Algorithm~\ref{alg:SP}, as stated in Proposition~\ref{prop:arcCosts} and Corollary~\ref{corr:correct} below.

\begin{proposition}\label{prop:arcCosts}
Given any pair of indices $i$ and $j$ corresponding to the outer (line~\ref{line:outer}) and inner (line~\ref{line:inner}) loops of Algorithm~\ref{alg:SP}, respectively, $\bar w=w_{ij}$ in line~\ref{line:barK}.
\end{proposition}
\begin{proof}
If $j=i+2$, then $\bar c=c_{i+1}$, $\bar q=Q_{i+1,i+1}$, $\bar w=a_{i+1}-\frac{c_{i+1}^2}{2Q_{i+1,i+1}}$ and the conclusion follows from Lemma~\ref{lem:elimination}. If $j=i+3$, then $\bar c=\tilde c_{i+2}$, $\bar q=\tilde Q_{i+2,i+2}$, and the conclusion follows from a recursive application of Lemma~\ref{lem:elimination} to the reduced problem, after the elimination of variable $x_{i+1}$. Similarly, cases $j>i+3$ follow from recursive applications of Lemma~\ref{lem:elimination}. 
\end{proof}

\begin{corollary}\label{corr:correct}
At the end of Algorithm~\ref{alg:SP}, label $\ell_k$ corresponds to the length of the shortest $(0,k)$-path. In particular, Algorithm~\ref{alg:SP} returns the length of the shortest $(0,n+1)$ path.
\end{corollary}
\begin{proof}
The proof follows due to the fact that line~\ref{line:labelUpdate} corresponds to the update of the shortest path labels using the natural topological order of \Gsp. 
\end{proof}

\begin{remark}
Algorithm~\ref{alg:SP} can be easily modified to recover, not only the optimal objective value, but also the optimal solution. This can be done by maintaining the list of predecessors $p$ of each node (initially, $p_k\leftarrow \emptyset$) throughout the algorithm; if label $\ell_j$ is updated at line~\ref{line:labelUpdate}, then set $p_j\leftarrow i$. The solution can then be recovered by backtracking, starting from $p_{n+1}$.\hfill \hfill \qed
\end{remark}

\subsection{Convexification}\label{sec:convexHull}

The polynomial time solvability of problem \eqref{eq:tridiagonalMIQO} suggests that it may be possible to find a tractable representation of the convex hull of $X$ when $Q$ is tridiagonal. Moreover, given a shortest path  (or, equivalently, dynamic programming) formulation of a pure integer linear optimization problem, it is often possible to construct an extended formulation of the convex hull of the feasible region, e.g., see \cite{lozano2020consistent,em87,w02,GK13}. 
There have been recent efforts to generalize such methods to nonlinear integer problems \cite{davarnia2021outer}, but few authors have considered using such convexification techniques in nonlinear mixed-integer problems as the ones considered here. Next, using lifting \cite{richard2010lifting} and the equivalence of optimization over $X$ to a shortest path problem proved in Section~\ref{sec:shortestPath}, we derive a compact extended formulation for $\text{cl conv}(X)$ in the tridiagonal case. 

The lifting approach used here is similar to the approach used recently in \cite{atamturk2020supermodularity,gomez2021outlier}: the continuous variables are projected out first, then a convex hull description is obtained for the resulting projection in the space of discrete variables, and finally the description is lifted back to the space of continuous variables. Unlike \cite{atamturk2020supermodularity,gomez2021outlier}, the convexification in the discrete space is obtained using an extended formulation (instead of finding the description in the original space of variables). 

In particular, to construct valid inequalities for $X$
 we observe that for any $(x,z,t)\in X$ and any $\theta\in \R^n$, 
\begin{align}
&\frac{1}{2}t\geq \frac{1}{2}x^\top Qx\notag\\
\Leftrightarrow &\frac{1}{2}t-\theta^\top x\geq -\theta^\top x+\frac{1}{2}x^\top Qx\notag\\
\implies &\frac{1}{2}t-\theta^\top x\geq g_\theta(z)\defeq \min_{x}\left\{-\theta^\top x+\frac{1}{2}x^\top Qx:x_i(1-z_i)=0,\; i\in N\right\}.\label{eq:valid0}
\end{align}

We now discuss the convexification in the space of the discrete variables, that is, describing the convex envelope of function $g_\theta$.

\subsubsection{Convexification of the projection in the $z$ space}

We study  the epigraph of function $g_\theta$, given by 
$$G_\theta=\left\{(z,s)\in  \{0,1\}^n\times \R: s\geq g_\theta(z)\right\}.$$ 
Note that $\conv(G_\theta)$ is polyhedral. Using the results from Section~\ref{sec:shortestPath}, we now give an extended formulation for $G_\theta$. Given two indices $0\leq i< j\leq n+1$ and vector $\theta\in \R^{n}$, define the function
\begin{align*}
    g_{ij}(\theta)&\defeq-\frac{1}{2}\theta[i,j]^\top Q[i,j]^{-1}\theta[i,j],
\end{align*}
Observe that $g_{ij}(\theta)=g_{ij}(-\theta)$ for any $\theta\in \R^n$, and that weights $w_{ij}$ defined in \eqref{eq:definition_w_ij} are given by $w_{ij}=g_{ij}(-c[i,j])+\sum_{k=i+1}^{j-1}a_k$.
Moreover, for $0\leq i<j\leq n+1$, consider variables $u_{ij}$ intuitively defined as ``$u_{ij}=1$ if and only if arc $(i,j)$ is used in a shortest $(0,n+1)$ path in \Gsp."  Consider the constraints
\begin{subequations}\label{eq:SP_extended}
\begin{align}
&\sum_{i=0}^n\sum_{j=i+1}^{n+1}g_{ij}(\theta)u_{ij}\leq s\label{eq:SP_extended_epi}\\
&\sum_{i=0}^{k-1}u_{ik}-\sum_{j=k+1}^{n+1}u_{kj}=\begin{cases}-1&\text{if }k=0\\
1&\text{if }k=n+1\\
0&\text{otherwise.}\end{cases}&k=0,\dots,n+1\label{eq:SP_extended_flow}\\
&\sum_{i=0}^{k-1}u_{ik}=1-z_k&k=1,\dots,n.\\
&u\geq 0, 0\leq z\leq 1. \label{eq:SP_extended_bounds}
\end{align}
\end{subequations}

\begin{proposition}\label{prop:extendedFormulation}
If $Q$ is tridiagonal, then the system \eqref{eq:SP_extended} is an extended formulation of $\conv(G_\theta)$ for any $\theta\in \R^n$.
\end{proposition}
\begin{proof}
It suffices to show that optimization over $G_{\theta}$ is equivalent to optimization over constraints \eqref{eq:SP_extended}. Optimization over $G_\theta$ corresponds to
\begin{align*}
&\min_{(z,s)\in G_\theta}\{\beta^\top z+s\}\\
\Leftrightarrow& \min_{x,z}\left\{-\theta^\top x+\beta^\top z+\frac{1}{2}x^\top Qx\right\}\;\text{  s.t.  }\;x_i(1-z_i)=0,\; z\in \{0,1\}^n
\end{align*}
for an arbitrary vector $\beta\in \R^n$. 
On the other hand, optimization over \eqref{eq:SP_extended} is equivalent, after projecting out variables $z$ and $s$, to 
\begin{align}
&\min_{u\geq 0}\sum_{i=1}^n\beta_i\left(1-\sum_{\ell=0}^{i-1}u_{\ell i}\right)+\sum_{i=1}^n\sum_{j=i+1}^{n+1}g_{ij}(\theta)u_{ij}\;\text{  s.t.  }\;\eqref{eq:SP_extended_flow},\eqref{eq:SP_extended_bounds}\notag\\
\Leftrightarrow&\min_{u\geq 0}\sum_{i=1}^n\beta_i\left(\sum_{\ell=0}^{i-1}\sum_{j=i+1}^{n+1}u_{\ell j}\right)+\sum_{i=1}^n\sum_{j=i+1}^{n+1}g_{ij}(\theta)u_{ij}\;\text{  s.t.  }\;\eqref{eq:SP_extended_flow},\eqref{eq:SP_extended_bounds}\label{eq:bypass}
\\
\Leftrightarrow&\min_{u\geq 0}\sum_{i=1}^n\sum_{j=i+1}^{n+1}\left(g_{ij}(\theta)+\sum_{\ell=i+1}^{j-1}\beta_\ell\right)u_{ij}\;\text{  s.t.  }\;\eqref{eq:SP_extended_flow},\eqref{eq:SP_extended_bounds},\notag
\end{align}
where \eqref{eq:bypass} follows from the observation that if node $i$ is not visited by the path ($\sum_{\ell=0}^{i-1}u_{\ell i}=0$), then one arc bypassing node $i$ is used. The equivalence between the two problems follows from Corollary~\ref{cor:shortestPath} and the fact that \eqref{eq:SP_extended_flow},\eqref{eq:SP_extended_bounds} are precisely the constraints corresponding to a shortest path problem. 
\end{proof}

\subsubsection{Lifting into the space of continuous variables}

From inequality \eqref{eq:valid0} and Proposition~\ref{prop:extendedFormulation}, we find that for any $\theta\in \R^n$ the linear inequality
\begin{equation}\label{eq:validLinear}\frac{1}{2}t-\theta^\top x\geq \sum_{i=0}^n\sum_{j=i+1}^{n+1}g_{ij}(\theta)u_{ij}\end{equation}
 is valid for $X$, where $(u,z)$ satisfy the constraints in \eqref{eq:SP_extended}. Of particular interest is choosing $\theta$ 
 that maximizes the strength of \eqref{eq:validLinear}:
 \begin{equation}
 \label{eq:validNonLinear}
 \frac{1}{2}t\geq \max_{\theta\in \R^n}\left\{ \sum_{i=0}^n\sum_{j=i+1}^{n+1}g_{ij}(\theta)u_{ij}+\theta^\top x\right\}.
\end{equation}

\begin{proposition}\label{prop:convexHull}If $Q$ is tridiagonal, then inequality \eqref{eq:validNonLinear} and constraints \eqref{eq:SP_extended_flow}-\eqref{eq:SP_extended_bounds} are sufficient to describe $\text{cl}\ \conv(X)$ (in an extended formulation). 
\end{proposition}
Proposition~\ref{prop:convexHull} is a direct consequence of Theorem~1 in \cite{richard2010lifting}. Nonetheless, for the sake of completeness, we include a short proof.
\begin{proof}[Proof of Proposition~\ref{prop:convexHull}]
	Consider the optimization problem \eqref{eq:miqo}, and its relaxation given by
	\begin{subequations}\label{eq:fullRelax}
	\begin{align}
	    \min_{x,z,u}\;&a^\top z+c^\top x+\max_{\theta\in \R^n}\left\{ \sum_{i=0}^n\sum_{j=i+1}^{n+1}g_{ij}(\theta)u_{ij}+\theta^\top x\right\}\\
	    \text{s.t.}\;&\eqref{eq:SP_extended_flow}-\eqref{eq:SP_extended_bounds}.
	\end{align}
	\end{subequations}
It suffices to show that there exists an optimal solution of \eqref{eq:fullRelax} which is feasible for \eqref{eq:miqo} with the same objective value, and thus it is optimal for \eqref{eq:miqo} as well. We consider a further relaxation of \eqref{eq:fullRelax}, obtained by fixing $\theta=-c$ in the inner maximization problem:
\begin{subequations}\label{eq:specificRelax}
	\begin{align}
	    \min_{x,z,u}\;&a^\top z+ \sum_{i=0}^n\sum_{j=i+1}^{n+1}g_{ij}(-c)u_{ij}\\
	    \text{s.t.}\;&\eqref{eq:SP_extended_flow}-\eqref{eq:SP_extended_bounds}.
	\end{align}
	\end{subequations}
 In particular, the objective value of \eqref{eq:specificRelax} is the same for all values of $x\in \R^n$. Moreover, using identical arguments to Proposition~\ref{prop:extendedFormulation}, we find that the two problems have in fact the same objective value. Finally, given any optimal solution $z^*$ to \eqref{eq:specificRelax}, the point $(x(z^*),z^*)$ is feasible for \eqref{eq:miqo} and optimal for its relaxation \eqref{eq:specificRelax} (with the same objective value), and thus this point is optimal for \eqref{eq:miqo} as well.
\end{proof}

We close this section by presenting an explicit form of inequalities \eqref{eq:validNonLinear}. Define $\bar Q(i,j)\in \R^{n\times n}$ as the matrix obtained by completing $Q[i,j]$ with zeros, that is, $\bar Q(i,j)[i,j]=Q[i,j]$ and $\bar Q(i,j)_{k\ell}=0$ otherwise. Moreover, by abusing notation, we define $\bar Q(i,j)^{-1}\in \R^{n\times n}$ similarly, that is, $\bar Q(i,j)^{-1}[i,j]=Q[i,j]^{-1}$ and $\bar Q(i,j)_{k\ell}^{-1}=0$ otherwise.
\begin{proposition}
For every $(x,u)\in \text{cl}\ \conv(X)$, inequality \eqref{eq:validNonLinear} is equivalent to 
\begin{equation}\label{eq:sdpRepresentation}
\begin{pmatrix}
t & x^\top\\
x & \left(\sum_{i=0}^n\sum_{j=i+1}^{n+1}\bar Q(i,j)^{-1}u_{ij}\right)
\end{pmatrix}\succeq 0. 
\end{equation}
\end{proposition}
\begin{proof}
Note that 
 \begin{align}
&\frac{1}{2}t\geq \max_{\theta\in \R^n}\left\{ \sum_{i=0}^n\sum_{j=i+1}^{n+1}g_{ij}(\theta)u_{ij}+\theta^\top x\right\}\notag\\
\Leftrightarrow\;&\frac{1}{2}t\geq \max_{\theta\in \R^n}\left\{ -\sum_{i=0}^n\sum_{j=i+1}^{n+1}\left(\frac{1}{2}\theta[i,j]^\top Q[i,j]^{-1}\theta[i,j]\right)u_{ij}+\theta^\top x\right\}\notag\\
\Leftrightarrow\;&\frac{1}{2}t\geq \max_{\theta\in \R^n}\left\{ -\frac{1}{2}\theta^\top\left(\underbrace{\sum_{i=0}^n\sum_{j=i+1}^{n+1}\bar Q(i,j)^{-1}u_{ij}}_{M(u)}\right)\theta+\theta^\top x\right\}.\label{eq:explicit}
\end{align}
Observe that matrix $M(u)$ is positive semidefinite (since it is a nonnegative sum of psd matrices), thus the maximization \eqref{eq:explicit} is a convex optimization problem, and its optimal value takes the form
\begin{align}
    \begin{cases}
    \frac{1}{2}x^\top M(u)^{\dagger}x & \text{if}\ \ {x\in \mathrm{Range}(M(u))},\\
    +\infty & \text{if \ \ otherwise},
    \end{cases}\nonumber
\end{align}
where $M(u)^{\dagger}$ and $\mathrm{Range}(M(u))$ denote the pseudo-inverse and range of $M(u)$, respectively. If $x\not\in \mathrm{Range}(M(u))$, then inequality \eqref{eq:validNonLinear} is violated, and hence, $(x,u)\not\in \text{cl}\ \conv(X)$. Therefore, we must have $x\in \mathrm{Range}(M(u))$, or equivalently, $M(u)M(u)^\dagger x = x$. In other words,
\begin{align}
    t\geq x^\top M(u)^{\dagger}x\qquad \text{and}\qquad {M(u)M(u)^\dagger x = x}.\nonumber
\end{align}
Invoking the Schur complement~\cite[Appendix A.5]{boyd2004convex} completes the proof.

\end{proof}

\section{General (Sparse) Graphs}\label{sec:fenchel}

In this section, we return  our attention to problem
\eqref{eq:miqo} where graph $\mathcal G$ is not a path (but is nonetheless assumed to be sparse), and matrix $Q$ is diagonally dominant: 
\begin{subequations}\label{eq:ddMiqoOrdered}
\begin{align}
    \zeta^*=\min_{x\in \R^n,z\in \{0,1\}^n}\;&a^\top z+c^\top x+\frac{1}{2}\sum_{i=1}^n  D_{ii}x_i^2+\frac{1}{2}\sum_{i=1}^n\sum_{j=i+1}^n|Q_{ij}|(x_i\pm x_{j})^2\label{eq:ddMiqoOrderedobj}\\
    \text{s.t.}\;&-Mz\leq x\leq Mz,
\end{align}
\end{subequations}
where $D_{ii} = Q_{ii}-\sum_{j\not= i}|Q_{ij}|\ge 0$. 
Note that \eqref{eq:map_miqo} is a special case of \eqref{eq:ddMiqoOrdered} with $n=|N|$, $c_i=-2y_i/\sigma_i^2$, $Q_{ij}=-2/d_{ij}$ if $(i,j)\in E$ and $Q_{ij}=0$ otherwise, $D_{ii}=2/\sigma_i^2$, and every ``$\pm$" sign corresponds to a minus sign. 

A natural approach to leverage the efficient $\mathcal{O}(n^2)$ algorithm for the tridiagonal case given in Section~\ref{sec:efficient_alg} is to simply drop terms $|Q_{ij}|(x_i\pm x_{j})^2$ whenever $j>i+1$, and solve the relaxation with objective 
$$a^\top z+c^\top x+\frac{1}{2}\sum_{i=1}^n  D_{ii}x_i^2+\frac{1}{2}\sum_{i=1}^n|Q_{i,i+1}|(x_i\pm x_{i+1})^2.$$
Intuitively, if matrix $Q$ is ``close" to tridiagonal, the resulting relaxation could be a close approximation to \eqref{eq:ddMiqoOrdered}. Nonetheless, as Example~\ref{ex:tree} below shows, the relaxation can in fact be quite loose. 

\begin{example}\label{ex:tree}

Consider the optimization problem with support graph given in Figure~\ref{fig:4pointgraph}:
\begin{subequations}\label{example1}
\begin{align}
    \zeta^*=\min &-1.3x_1-2.5x_2+4.6x_3-7.8x_4+3x_1^2+6x_2^2+3x_3^2+2x_4^2\notag\\
    &-1.5x_1x_2-x_2x_3-0.8x_2x_4+2(z_1+z_2+z_3+z_4)\label{example1_obj}\\
    \text{s.t.}\;& x_i(1-z_i)=0,\;z_i\in\{0,1\}\qquad i=1,\dots,4.
\end{align}
\end{subequations}

\begin{figure}[htbp]
\centering
\includegraphics[width=0.6\textwidth]{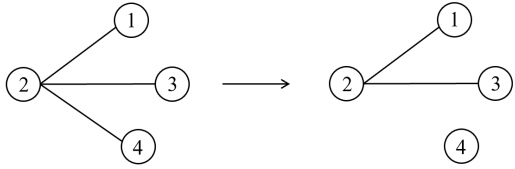}
\caption{Support graphs for Example~\ref{ex:tree}. Left: original; right: after dropping $0.4(x_2-x_4)^2$.}
\label{fig:4pointgraph}
\end{figure}

The optimal solution of \eqref{example1} is $x^*=(0,0,-1.53,3.9)$ with objective value $\zeta^*\approx-14.74$. After deletion of the term $0.4(x_2-x_4)^2$ from \eqref{example1_obj},  we obtain the tridiagonal problem
\begin{subequations}\label{0relaxation}
\begin{align*}
    \zeta_0=\min &-1.3x_1-2.5x_2+4.6x_3-7.8x_4+3x_1^2+5.6x_2^2+3x_3^2+1.6x_4^2\notag \\
    &-1.5x_1x_2-x_2x_3+2(z_1+z_2+z_3+z_4)\\
    \text{s.t.}\;& x_i(1-z_i)=0,\;z_i\in\{0,1\}\qquad i=1,\dots,4,
\end{align*}
\end{subequations}
with optimal solution $x^*_0=(0,0,-1.53,6.5)$ and $\zeta_0=-24.88$. The optimality gap is $\frac{\zeta^*-\zeta_0}{|\zeta^*|}=68.8\%$. \hfill \hfill \qed
\end{example}
We now discuss how to obtain improved relaxations of \eqref{eq:ddMiqoOrdered}, which can result in much smaller optimality gaps.

\subsection{Convex relaxation via path and rank-one convexifications}

The large optimality gap in Example~\ref{ex:tree} can be attributed to the large effect of completely ignoring some terms $|Q_{ij}|(x_i\pm x_{j})^2$. To obtain a better relaxation, we use the following  convexification of rank-one terms for set 
$$X_{2}=\left\{(x,z,t)\in \R^2\times \{0,1\}^2\times \R: (x_1\pm x_2)^2\leq t,\; x_{i}(1-z_i)=0,\; i=1,2\right\}.$$

\begin{proposition}[Atamt\"urk and G\'omez 2019 \cite{atamturk2019rank}]\label{prop:rank-one}
$$\text{cl }\conv(X_{2})=\left\{(x,z,t)\in \R^2\times [0,1]^2\times \R: \frac{(x_1\pm x_2)^2}{\min\{1,z_1+z_2\}}\leq t\right\}.$$
\end{proposition}

We propose to solve the convexification of \eqref{eq:ddMiqoOrdered} where binary variables are relaxed to $0\leq z\leq 1$, complementarity constraints are removed, each term $|Q_{ij}|(x_i\pm x_{j})^2$ with $j>i+1$ is replaced with the convexification in Proposition~\ref{prop:rank-one}, and the rest of the terms are convexified using the results of Section~\ref{sec:convexHull}. Formally, define the ``zero"-indices $1=\tau_1<\dots<\tau_{\ell}<\tau_{\ell+1}=n+1$  such that 
$$\{\tau_2,\dots,\tau_{\ell}\}=\left\{i\in N:Q_{i-1,i}=0\right\}.$$ 
Moreover, given indices $0\leq i<j\leq n+1$, define \begin{align*}X(i,j)=\Big\{(x,z,t)\in& \R^{j-i-1}\times\{0,1\}^{j-i-1}\times \R: t\geq x^\top \hat Q(i,j)x,\\
&x_t(1-z_t)=0, \;t=1,\dots,j-i-1 \Big\},\end{align*}
where $\hat Q(i,j)$ is the tridiagonal $(j-i-1)\times(j-i-1)$ matrix corresponding to indices $i+1$ to $j-1$ in problem \eqref{eq:ddMiqoOrdered}, that is, $$\hat Q(i,j)_{t k}=\begin{cases}D_{t+i, t+i}+Q_{t+i,t+i+1}&\text{if }t=k\\
Q_{t+i, k+i} &\text{if }t-k= \pm 1\\
0&\text{otherwise}.\end{cases}$$
A description of $\text{cl }\conv(X(i,j))$ is given in Proposition~\ref{prop:convexHull}. We propose to solve the relaxation of \eqref{eq:ddMiqoOrdered} given by 
\begin{subequations}\label{eq:convexification}
\begin{align}
    \zeta_{p}=\min_{x,z,t}\;&a^\top z+c^\top x+\frac{1}{2}\sum_{k=1}^{\ell}t_k+
    \frac{1}{2}\sum_{i=1}^n\sum_{j=i+2}^n|Q_{i,j}|\frac{(x_i\pm x_{j})^2}{\min\{1,z_i+z_j\}}\\
    \text{s.t.}\;&(x[\tau_{k}-1,\tau_{k+1}],z[\tau_{k}-1,\tau_{k+1}],t_k)\in \text{cl}\ \conv(X(\tau_{k}-1,\tau_{k+1}))\notag\\
    &\hspace{6cm}  k=1,\dots,\ell\label{eq:convexification_sdp}\\
    &x\in \R^n, z\in [0,1]^n, t\in \R^{\ell}.
\end{align}
\end{subequations}
Note that the strength of relaxation \eqref{eq:convexification} depends on the order in which variables $x_1,\dots,x_n$ are indexed. We discuss how to choose an ordering in Section~\ref{sec:decomposition}. In the rest of this section, we assume that the order is fixed. 

We first establish that relaxation \eqref{eq:convexification} is stronger than the pure rank-one relaxation 
\begin{align}\label{eq:ddMiqoOrdered_R1}
    \zeta_{R1}=\min_{x\in \R^n,z\in [0,1]^n}\;&a^\top z+c^\top x+\frac{1}{2}\sum_{i=1}^n  D_{ii}\frac{x_i^2}{z_i}+\frac{1}{2}\sum_{i=1}^n\sum_{j=i+1}^n|Q_{ij}|\frac{(x_i\pm x_{j})^2}{\min\{1,z_i+z_j\}}
\end{align}
used in \cite{atamturk2019rank}.
\begin{proposition}
Given any vectors $a,c$ and matrix $Q$, $\zeta_{R_1}\leq \zeta_p\leq \zeta^*$.
\end{proposition}
\begin{proof}
Since both $\eqref{eq:convexification}$ and \eqref{eq:ddMiqoOrdered_R1} are relaxations of \eqref{eq:ddMiqoOrdered}, it follows that $\zeta_{R1}\leq \zeta^*$ and $\zeta_{p}\leq \zeta^*$. Moreover, note that 
\begin{align*}
    \min_{x,z,t}\;&a^\top z+c^\top x+\frac{1}{2}\sum_{k=1}^{\ell}t_k\\
    \text{s.t.}\;&(x[\tau_{k}-1,\tau_{k+1}],z[\tau_{k}-1,\tau_{k+1}],t_k)\in \text{cl}\ \conv(X(\tau_{k}-1,\tau_{k+1}))\notag\\
    &\hspace{6cm}  k=1,\dots,\ell\\
    &x\in \R^n, z\in [0,1]^n, t\in \R^{\ell}
\end{align*}
is an \emph{ideal} relaxation of the discrete problem 
\begin{align*}
\min_{x\in \R^n,z\in \{0,1\}^n}\;&a^\top z+c^\top x+\frac{1}{2}\sum_{i=1}^n  D_{ii}x_i^2+\frac{1}{2}\sum_{i=1}^n|Q_{i,i+1}|(x_i\pm x_{i+1})^2\\
    \text{s.t.}\;&-Mz\leq x\leq Mz,
\end{align*}
whereas \eqref{eq:ddMiqoOrdered_R1} 
is not necessarily an ideal relaxation of the same problem. Since the two relaxations coincide in how they handle terms $|Q_{ij}|(x_i\pm x_j)^2$ for $j>i+1$, it follows that $\zeta_{R1}\leq \zeta_{p}$.
\end{proof}

While relaxation \eqref{eq:convexification} is indeed strong and is conic-representable, it may be difficult to solve using off-the-shelf solvers. Indeed, the direct implementation of \eqref{eq:convexification_sdp} requires the addition of $\mathcal{O}(n^2)$ additional variables \emph{and} the introduction of the positive-semidefinite constraints \eqref{eq:sdpRepresentation}, resulting in a large-scale SDP. We now develop a tailored decomposition algorithm to solve \eqref{eq:convexification} based on Fenchel duality.

Decomposition methods for mixed-integer \emph{linear} programs, such as Lagrangian decomposition, have been successful in solving large-scale instances. In these frameworks, typically a set of ``complicating" constraints that tie a large number of variables are relaxed using a Lagrangian relaxation scheme. In this vein, in
  \cite{Fang2019}, the authors propose a Lagrangian relaxation-based decomposition method for  spatial graphical model estimation problems under cardinality constraints. The Lagrangian relaxation of the cardinality constraint results in a Lagrangian dual problem  that decomposes into smaller  mixed-integer subproblems, which are solved using optimization solvers. In contrast, in this paper, we use Fenchel duality to relax the ``complicating" terms in the objective. This results in subproblems that can be solved in polynomial time, in parallel. Furthermore, the strength of the Fenchel dual results in a highly scalable algorithm that converges to an optimal solution of the relaxation fast. Before we describe the algorithm, we first introduce the Fenchel dual problem.

\subsection{Fenchel duality}\label{sec:fenchel_dual}
The decomposition algorithm we propose relies on the Fenchel dual of terms resulting from the rank-one convexification.
\begin{proposition}[Fenchel dual] \label{prop:fenchel}For all $(x,z)$ satisfying $0\leq z\leq 1$ and each $\alpha,\beta_1,\beta_2\in\mathbb{R}$
\begin{equation}\label{fencheldual}
    \frac{(x_1 \pm x_2)^2}{\min\{1,z_1+z_2\}} \geq \alpha (x_1\pm x_2)-\beta_1z_1-\beta_2z_2-f^*(\alpha,\beta_1,\beta_2), 
\end{equation}
where \begin{align*}f^*(\alpha,\beta_1,\beta_2)\defeq& \max_{x,z}\;\alpha(x_1\pm x_2)-\beta_1z_1-\beta_2z_2-\frac{(x_1\pm x_2)^2}{\min\{1,z_1+z_2\}}\\
=&\max\{0,\frac{\alpha^2}{4}-\min\{\beta_1,\beta_2\}\}-\min\{\max\{\beta_1,\beta_2\},0\}.\end{align*}

Moreover, for any fixed $(x,z)$, there exists  $\alpha,\beta_1,\beta_2$ such that the inequality is tight.
\end{proposition}

\begin{proof}
For simplicity, we first assume that the ``$\pm$" sign is a minus. Define function $f$ as
\begin{equation*}
f(\alpha_1,\alpha_2,\beta_1,\beta_2,x_1,x_2,z_1,z_2)=\alpha_1x_1+\alpha_2x_2-\beta_1z_1-\beta_2z_2-\frac{{(x_1-x_2)}^2}{\min\{ 1,z_1+z_2\}},\\
\end{equation*}
and consider the maximization problem given by
\begin{subequations} \label{fenchelmax}
    \begin{align}
    f^*(\alpha_1,\alpha_2,\beta_1,\beta_2)=
    &\max_{x,z}\  f(\alpha_1,\alpha_2,\beta_1,\beta_2,x_1,x_2,z_1,z_2)\quad\\
    &\text{s.t.}\;x_1,x_2\in\mathbb{R},z_1, z_2 \in [0,1].
    \end{align}
\end{subequations}
We now compute an explicit form of $f^*(\alpha_1,\alpha_2,\beta_1,\beta_2)$. First, observe that if $\alpha_1+\alpha_2 \neq 0$, then $f^*(\alpha_1,\alpha_2,\beta_1,\beta_2) = +\infty$,   (with $z_1=z_2=1,x_1=x_2=\pm\infty$). Thus we assume without loss of generality that $\alpha_1+\alpha_2 = 0$, let $\alpha=\alpha_1$ and use the short notation $f^*(\alpha,\beta_1,\beta_2)$ and $ f(\alpha,\beta_1,\beta_2,x_1,x_2,z_1,z_2)$ instead of $f^*(\alpha,-\alpha,\beta_1,\beta_2)$ and $f(\alpha,-\alpha,\beta_1,\beta_2,x_1,x_2,z_1,z_2)$, respectively. 

To compute a maximum of \eqref{fenchelmax}, we consider three classes of candidate solutions. 
\begin{enumerate}
    \item Solutions with $z_1=z_2=0$. If $x_1=x_2$, then $f(\alpha,\beta_1,\beta_2,x_1,x_2,z_1,z_2)=0$. Otherwise, if $x_1\neq x_2$, then $f(\alpha,\beta_1,\beta_2,x_1,x_2,z_1,z_2)=-\infty$. 
    \item Solutions with $0<z_1+z_2\leq1$. 
    We claim that we can assume without loss of generality that $z_1+z_2=1$. First, if $z_1=0$, then the objective is homogeneous in $z_2$, thus there exists an optimal solution where either $z_2=0$ (but this case has already been considered) or $z_2=1$. The case $z_2=0$ is identical. Finally, if both $0<z_1,z_2$ and $z_1+z_2<1$, then from the optimality conditions we find that
    $$0=-\beta_1+\left(\frac{x_1-x_2}{z_1+z_2}\right)^2=-\beta_2+\left(\frac{x_1-x_2}{z_1+z_2}\right)^2,$$
    and in particular this case only happens if $\beta_1=\beta_2$. In this case the objective is homogeneous in $z_1+z_2$, 
thus there exists another optimal solution where either $z_1+z_2\to 0$ (already considered) or $z_1+z_2=1$. 

Moreover, in the $z_1+z_2=1$ case,
    \begin{align*}
    f(\alpha,\beta_1,\beta_2,x_1,x_2,z_1,z_2)&=\alpha(x_1-x_2)-{(x_1-x_2)}^2-\beta_1z_1-\beta_2z_2\\
    &\leq\frac{\alpha^2}{4}-\beta_1z_1-\beta_2z_2\tag{equal if $x_1-x_2=\alpha/2$}\\
    &\leq\frac{\alpha^2}{4}-\min\{\beta_1,\beta_2\}.
    \end{align*}
    \item Solutions with $z_1+z_2>1$. In this case, the objective is linear in $z$, thus in an optimal solution, $z$ is at its bound. The only case not considered already is $z_1=z_2=1$, where
    \begin{align*}
    f(\alpha,\beta_1,\beta_2,x_1,x_2,z_1,z_2)&=\alpha(x_1-x_2)-{(x_1-x_2)}^2-\beta_1z_1-\beta_2z_2\\
    &\leq\frac{\alpha^2}{4}-\beta_1-\beta_2.\tag{equal if $x_1-x_2=\alpha/2$}
    \end{align*}
\end{enumerate}
Thus, to compute an upper bound on $f^*(\alpha,\beta_1,\beta_2)$, it suffices to compare the three values $0$, $\alpha^2/4-\min\{\beta_1,\beta_2\}$, and $\alpha^2/4-\beta_1-\beta_2$ and choose the largest one. 
The result can be summarized as
\begin{equation*}
    f^*(\alpha,\beta_1,\beta_2)\leq\max\{0,\frac{\alpha^2}{4}-\min\{\beta_1,\beta_2\}\}-\min\{\max\{\beta_1,\beta_2\},0\}.
\end{equation*}

Finally, for a given $(x,z)$, we discuss how to choose $(\alpha,\beta_1,\beta_2)$ so that inequality \eqref{fencheldual} is tight. 
\begin{itemize}
    \item[$\bullet$]  If $z=0$ and $x_1-x_2=0$, then set $\alpha=\beta_1=\beta_2=0$.
    \item[$\bullet$] If $z=0$ and $x_1-x_2\neq 0$, then set $\alpha=\rho (x_1-x_2)$ with $\rho\to \infty$, and set $\beta_1=\beta_2=\alpha^2/4$.
    \item[$\bullet$] If $0<z_1+z_2<1$, then set $\alpha=2 (x_1-x_2)$ and $\beta_1=\beta_2=(x_1-x_2)^2/(z_1+z_2)$.
    \item[$\bullet$] If $1\leq z_1+z_2$, then set $\alpha=2 (x_1-x_2)$ and $\beta_1=\beta_2=0$. 
\end{itemize}

\end{proof}

\begin{remark}\label{rem:fenchelCases}
Function $f^*$ is defined by pieces, corresponding to a maximum of convex functions. By analyzing under which cases the maximum is attained, we find that
$$f^*(\alpha,\beta_1,\beta_2)=\begin{cases}0&\text{if }\min\{\beta_1,\beta_2\}>\alpha^2/4\\
\alpha^2/4-\beta_1-\beta_2&\text{if }\max\{\beta_1,\beta_2\}<0\\
\alpha^2/4-\min\{\beta_1,\beta_2\}&\text{otherwise}.
\end{cases}$$
Indeed:
\begin{enumerate}
    \item Case $\min\{\beta_1,\beta_2\}>\alpha^2/4$. Since $\alpha^2/4-\beta_1-\beta_2<\alpha^2/4-\min\{\beta_1,\beta_2\}<0$, we conclude that $f^*(\alpha,\beta_1,\beta_2)=0$ with $x_1^*,x_2^*,z_1^*,z_2^*=0$.
     \item Case $\max\{\beta_1,\beta_2\}<0$. Since $0<\alpha^2/4-\min\{\beta_1,\beta_2\}<\alpha^2/4-\beta_1-\beta_2$,  we conclude that $f^*(\alpha,\beta_1,\beta_2)=\alpha^2/4-\beta_1-\beta_2$, with $x_1^*-x_2^*=\alpha/2,z_1^*=z_2^*=1$.
    \item Case $\min\{\beta_1,\beta_2\}\leq\alpha^2/4$ and $\max\{\beta_1,\beta_2\}\geq0$. Since $0,  \alpha^2/4-\beta_1-\beta_2\leq \alpha^2/4-\min\{\beta_1,\beta_2\}$, we conclude that $f^*(\alpha,\beta_1,\beta_2)=\alpha^2/4-\min\{\beta_1,\beta_2\}$ with $x_1^*-x_2^*=\alpha/2,z_1^*+z_2^*=1$.\hfill \hfill \qed
\end{enumerate}
\end{remark}

From Proposition~\ref{prop:fenchel}, it follows that problem \eqref{eq:convexification} can be written as \begin{subequations}\label{eq:convexification_fenchel}
\begin{align}
    \zeta_{p}=\min_{x,z,t}\;&\max_{\alpha,\beta}\;a^\top z+c^\top x+\frac{1}{2}\sum_{k=1}^{\ell}t_k\notag\\
    &\hspace{-0.9cm}+
    \frac{1}{2}\sum_{i=1}^n\sum_{j=i+2}^n|Q_{ij}|\Big(\alpha_{ij}(x_i\pm x_j)-\beta_{ij,i}z_i-\beta_{ij,j}z_j-f^*(\alpha_{ij},\beta_{ij,i},\beta_{ij,j})\Big)\\
    \text{s.t.}\;&(x[\tau_{k}-1,\tau_{k+1}],z[\tau_{k}-1,\tau_{k+1}],t_k)\in \text{cl}\ \conv(X(\tau_{k}-1,\tau_{k+1}))\notag\\
    &\hspace{6cm}  k=1,\dots,\ell\\
    &(x,z)\in \mathcal{X}, t\in \R^{\ell},
\end{align}
\end{subequations}
where $\mathcal{X}=\{(x,z)\in \R^n\times [0,1]^n: -Mz\leq x\leq Mz\}$. Define
\begin{align*}\psi(x,z,t;\alpha,\beta)\defeq&\sum_{i=1}^n\left(a_i-\frac{1}{2}\sum_{j=i+2}^n|Q_{ij}|\beta_{ij,i}-\frac{1}{2}\sum_{j=1}^{i-2}|Q_{ji}|\beta_{ji,i}\right)z_i\notag\\
+&\sum_{i=1}^n\left(c_i+\frac{1}{2}\sum_{j=i+2}^n|Q_{ij}|\alpha_{ij}+\frac{1}{2}\sum_{j=1}^{i-2}(\pm|Q_{ji}|\alpha_{ji})\right)x_i+\frac{1}{2}\sum_{k=1}^{\ell}t_k.
\end{align*}

\begin{proposition}[Strong duality]\label{prop:strong duality}
Strong duality holds for problem \eqref{eq:convexification_fenchel}, that is, 
\begin{align*}
    \zeta_{p}=\max_{\alpha,\beta}\;h(\alpha,\beta),
\end{align*}
where
{
\begin{subequations}\label{eq:convexification_dual}
\begin{align}
    h(\alpha,\beta)\defeq &-\frac{1}{2}\sum_{i=1}^n\sum_{j=i+2}^n|Q_{ij}| f^*(\alpha_{ij},\beta_{ij,i},\beta_{ij,j})
+\min_{x,z,t}\;\Big\{\psi(x,z,t,\alpha,\beta)\label{eq:convexification_obj}\Big\}\\
    \text{s.t.}\;&(x[\tau_{k}-1,\tau_{k+1}],z[\tau_{k}-1,\tau_{k+1}],t_k)\in \text{cl}\ \conv(X(\tau_{k}-1,\tau_{k+1}))\notag\\
    &\hspace{6cm}  k=1,\dots,\ell\\
    &(x,z)\in \mathcal{X}, t\in \R^{\ell}.
\end{align}
\end{subequations}
}

\end{proposition}
\begin{proof}
Problem \eqref{eq:convexification_dual} is obtained from \eqref{eq:convexification_fenchel} by interchanging $\min$ and $\max$. Equality between the two problems follows from Corollary 3.3 in \cite{Sion58}, because $f^*$ is convex (since by definition it is a supremum of affine functions) and $\mathcal{X}$ is compact. 
\end{proof}

\subsection{Subgradient algorithm}\label{subsec3.2}

Note that for any fixed $(\alpha,\beta)$, the inner minimization problem in \eqref{eq:convexification_dual} decomposes into independent subproblems, each involving variables \linebreak $(x[\tau_{k}-1, \tau_{k+1}],  z[\tau_{k}-1,\tau_{k+1}], t_k)$. Moreover, each subproblem corresponds to an optimization problem over $\text{cl }\conv(X(\tau_{k}-1,\tau_{k+1}))$, equivalently, optimization over $X(\tau_{k}-1,\tau_{k+1})$. Therefore, it can be solved in $\mathcal{O}(n^2)$ time using Algorithm~\ref{alg:SP}. Furthermore, the outer maximization problem is concave in $(\alpha,\beta)$ since $h(\alpha,\beta)$ is the sum of the concave functions $- f^*(\alpha_{ij},\beta_{ij,i},\beta_{ij,j})$ and an infimum of affine functions, thus it can in principle be optimized efficiently. We now discuss how to solve the latter problem via a subgradient method.

Similar to Lagrangian decomposition methods for mixed-integer linear optimization \cite{wolsey1999integer}, subgradients of function $h$ can be obtained directly from optimal solutions of the inner minimization problem. Given any point $(\bar \alpha$, $\bar \beta_{1},\bar \beta_{2})\in \R^3$, denote by $\partial f^*(\bar \alpha, \bar \beta_{1},\bar \beta_{2}))$ the subdifferential of $f^*$ at that point. In other words, $(\xi(\bar \alpha),\xi(\bar \beta_{1}),\xi(\bar \beta_{2})))\in \partial f^*(\bar \alpha, \bar \beta_{1},\bar\beta_{2}))$ implies \begin{align}& \notag\\
 f^*(\alpha,\beta_{1},\beta_{2})\geq f^*(\bar \alpha,\bar \beta_{1},\bar \beta_{2})+\xi(\bar \alpha)(\alpha-\bar \alpha)+\xi(\bar \beta_{1})(\beta_1-\bar \beta_1)+\xi(\bar \beta_{2})(\beta_2-\bar \beta_2)\label{eq:subgradientF}
\end{align}
for all $(\alpha,\beta_1,\beta_2)\in \R^3$. The next proposition  shows that subgradients $\rho(\alpha,\beta)\in \R^{(3/2)(n-1)(n-2)}$ of function $h(\alpha,\beta)$ (for maximization) can be obtained from subgradients of $f^*$ and optimal solutions $(\bar x,\bar z)$ of the inner minimization problem in \eqref{eq:convexification_dual}, as
$$\begin{pmatrix}\rho(\alpha_{ij})\\
\rho(\beta_{ij,i})\\
\rho(\beta_{ij,j})
\end{pmatrix}=-\begin{pmatrix}\xi(\alpha_{ij})\\
\xi(\beta_{ij,i})\\
\xi(\beta_{ij,j})
\end{pmatrix}+\begin{pmatrix}
\bar x_i\pm\bar x_j\\
-\bar z_i\\
-\bar z_j
\end{pmatrix}.$$
Later, in Proposition~\ref{prop:subgradientF}, we explicitly describe the subgradients $\xi$ of $f^*$. 

\begin{proposition}\label{prop:subgradientH}
Given any $(\bar\alpha,\bar\beta)\in \R^{(3/2)(n-1)(n-2)}$, let $(\bar x,\bar z,\bar t)$ denote an optimal solution of the associated minimization problem \eqref{eq:convexification_dual}, and let $$(\xi(\bar \alpha_{ij}),\xi(\bar \beta_{ij,i}),\xi(\bar \beta_{ij,j}))\in \partial f^*(\bar \alpha_{ij}, \bar \beta_{ij,i},\bar\beta_{ij,j})$$ for all $j>i+1$. Then for any  $(\alpha,\beta)\in \R^{(3/2)(n-1)(n-2)}$, 
\begin{align*}h(\alpha,\beta)&\leq h(\bar\alpha,\bar\beta)+\frac{1}{2}\sum_{i=1}^n\sum_{j=i+2}^n|Q_{ij}|\left(-\xi(\bar \alpha_{ij})+\bar x_i\pm \bar x_j\right)(\alpha_{ij}-\bar \alpha_{ij})\\
+ &\frac{1}{2}\sum_{i=1}^n\sum_{j=i+2}^n |Q_{ij}| \left(-\xi(\bar \beta_{ij,i})-\bar z_i\right)(\beta_{ij,i}-\bar \beta_{ij,i})\\
+ &\frac{1}{2}\sum_{i=1}^n\sum_{j=i+2}^n |Q_{ij}| \left(-\xi(\bar \beta_{ij,j})-\bar z_j\right)(\beta_{ij,j}-\bar \beta_{ij,j}).
\end{align*}
\end{proposition}
\begin{proof}
Given $(\alpha,\beta)\in \R^{(3/2)(n-1)(n-2)}$, let $(x^*,y^*,t^*)$ be the associated solution of the inner minimization problem \eqref{eq:convexification_dual}. 
Then we deduce that
{\small
\begin{align*}
    h(\alpha,\beta)=&-\frac{1}{2}\sum_{i=1}^n\sum_{j=i+2}^n|Q_{ij}| f^*(\alpha_{ij},\beta_{ij,i},\beta_{ij,j})+\psi(x^*,z^*,t^*; \alpha,\beta) \\
    \leq& -\frac{1}{2}\sum_{i=1}^n\sum_{j=i+2}^n|Q_{ij}|\Bigg( f^*(\bar\alpha_{ij},\bar\beta_{ij,i},\bar\beta_{ij,j})+\xi(\bar \alpha_{ij})(\alpha_{ij}-\bar \alpha_{ij})+\xi(\bar \beta_{ij,i})(\beta_{ij,i}-\bar \beta_{ij,i})\\
    &\qquad+\xi(\bar \beta_{ij,j})(\beta_{ij,j}-\bar \beta_{ij,j})\Bigg)
    +\psi(x^*,z^*,t^*; \alpha,\beta)\\
    \leq& -\frac{1}{2}\sum_{i=1}^n\sum_{j=i+2}^n|Q_{ij}|\Bigg( f^*(\bar\alpha_{ij},\bar\beta_{ij,i},\bar\beta_{ij,j})+\xi(\bar \alpha_{ij})(\alpha_{ij}-\bar \alpha_{ij})+\xi(\bar \beta_{ij,i})(\beta_{ij,i}-\bar \beta_{ij,i})\\
    &\qquad+\xi(\bar \beta_{ij,j})(\beta_{ij,j}-\bar \beta_{ij,j})\Bigg)
    +\psi(\bar x,\bar z,\bar t; \alpha,\beta)\\
     =& -\frac{1}{2}\sum_{i=1}^n\sum_{j=i+2}^n|Q_{ij}|\Bigg( f^*(\bar\alpha_{ij},\bar\beta_{ij,i},\bar\beta_{ij,j})+\big(\xi(\bar \alpha_{ij})+\bar x_i\pm \bar x_j\big)(\alpha_{ij}-\bar \alpha_{ij})\\
    &\qquad+\big(\xi(\bar \beta_{ij,i})-\bar z_i\big)(\beta_{ij,i}-\bar \beta_{ij,i})
    +\big(\xi(\bar \beta_{ij,j})-\bar z_j\big)(\beta_{ij,j}-\bar \beta_{ij,j})\Bigg)
    +\psi(\bar x,\bar z,\bar t; \bar\alpha,\bar\beta),
\end{align*}}
where the first inequality follows from \eqref{eq:subgradientF}, the second inequality follows since $(x^*, z^*,t^*)$ is a minimizer of $\psi(\cdot, \alpha,\beta)$ whereas $(\bar x, \bar z,\bar t)$ may not be, and the last equality follows since $\psi$ is linear in $(\alpha,\beta)$ for fixed $(\bar x, \bar z,\bar t)$. The conclusion follows. 
\end{proof}

\begin{proposition}\label{prop:subgradientF}
A subgradient of function $f^*$ admits a closed form solution as
\begin{align*}
    &(\xi(\alpha_{ij}),\xi(\beta_{ij,i}),\xi(\beta_{ij,j}))\\
    &=\left\{
\begin{array}{l@{\quad}l}
(0,0,0) &\text{if }\beta_{ij,i},\beta_{ij,j}>\alpha_{ij}^2/4,\\
(\alpha_{ij}/2,-1,0)  &\text{if }\beta_{ij,i}\leq\alpha_{ij}^2/4\ \&\ \beta_{ij,j}\geq0\text\ \&\ \beta_{ij,j}\geq\beta_{ij,i},\\
(\alpha_{ij}/2,0,-1)  &\text{if }\beta_{ij,j}\leq\alpha_{ij}^2/4\ \&\ \beta_{ij,i}\geq0\ \&\ \beta_{ij,i}>\beta_{ij,j},\\
(\alpha_{ij}/2,-1,-1) &\text{if }\beta_{ij,i},\beta_{ij,j}<0.\\
\end{array}
\right.
\end{align*}
\end{proposition}
\begin{proof}
The result follows since $f^*$ is the supremum of affine functions described in Remark~\ref{rem:fenchelCases}. 
\end{proof}

Algorithm~\ref{alg:Decomposition} states the proposed method to solve problem \eqref{eq:convexification_dual}. Initially, $(\alpha,\beta)$ is set to zero (line~\ref{line:decomposition_init}). Then, at each iteration: a primal solution $(\bar x,\bar z,\bar t)$ of \eqref{eq:convexification_dual} is obtained  by solving $\ell$ independent problems of the form \eqref{eq:tridiagonalMIQO}, using Algorithm~\ref{alg:SP} (line~\ref{line:decomposition_primal}); a subgradient of function $h$ is obtained directly from $(\bar x,\bar z)$ using Propositions~\ref{prop:subgradientH} and \ref{prop:subgradientF} (line \ref{line:decomposition_subgradient}); finally, the dual solution is updated using first order information (line~\ref{line:decomposition_dual}). 

\begin{algorithm}[h]
	\caption{Subgradient ascent }	
	\label{alg:Decomposition} 
	\begin{algorithmic}[1]
	\renewcommand{\algorithmicrequire}{\textbf{Input:}}
	\renewcommand{\algorithmicensure}{\textbf{Output:}}
	\Require $a,c\in \R^n$, $Q\succ 0$ diagonally dominant.
	\Ensure $\zeta_p$.
		\State $(\alpha,\beta)\leftarrow (0,0)$ \Comment{Initialize}\label{line:decomposition_init}
		\State $k\leftarrow 1$ \Comment{Iteration counter}
		\Repeat
		\State $(\bar x,\bar z,\bar t)\in {\arg\min}_{x,z,t}\psi(x,z,t;\alpha,\beta)$\Comment{Algorithm~\ref{alg:SP}} \label{line:decomposition_primal}
		\State $\rho(\bar x,\bar z)\in \partial h(\alpha,\beta)$\Comment{Propositions~\ref{prop:subgradientH} and \ref{prop:subgradientF}} \label{line:decomposition_subgradient}
		\State $(\alpha,\beta)\leftarrow (\alpha,\beta)+s_k \rho(\bar x,\bar z)/\|\rho(\bar x,\bar z)\|_2$ \Comment{$s_k=$Step size at iteration $k$} \label{line:decomposition_dual}
		\State $k\leftarrow k+1$ \Comment{Iteration counter}
		\Until{Termination criterion is met}\label{line:decomposition_termination}
		\State \Return $h(\alpha,\beta)$
	\end{algorithmic}
\end{algorithm}

We now discuss some implementation details. First, note that in line~\ref{line:decomposition_init}, $(\alpha,\beta)$ can in fact be initialized to an arbitrary point without affecting correctness of the algorithm. Nonetheless, by initializing at zero, we ensure that the first iteration of Algorithm~\ref{alg:Decomposition} corresponds to solving the relaxation obtained by completely dropping the complicating quadratic terms; see Example~\ref{ex:tree}. Second, each time a primal solution is obtained (line~\ref{line:decomposition_primal}), a lower bound $h(\alpha,\beta)\leq\zeta_p\leq \zeta^*$ can be computed. Moreover, since the solution $(\bar x,\bar z)$ is feasible for \eqref{eq:ddMiqoOrdered}, an upper bound $\bar \zeta(\bar x,\bar z)\geq \zeta^*$ can be obtained by simply evaluating the objective function \eqref{eq:ddMiqoOrderedobj}. Thus, at each iteration of Algorithm~\ref{alg:Decomposition}, an estimate of the optimality gap of $(\bar x,\bar z)$ can be computed as $\texttt{gap}=(\bar \zeta(\bar x,\bar z)-h(\alpha,\beta))/\bar \zeta(\bar x,\bar z)$. Third, a natural termination criterion (line~\ref{line:decomposition_termination}) we use in our experiments is to terminate if $\texttt{gap}\leq \epsilon$ for some predefined optimality tolerance $\epsilon$, or if $k\geq \bar k$ for some iteration limit $\bar k$. Observe that a termination criterion in addition to the optimality gap is indeed required: if $\zeta_p+\epsilon<\zeta^*$, then the estimated optimality gap will never be less than $\epsilon$. 
Nonetheless, it is worth pointing out that under mild conditions (such as diminishing nature of the step sizes), the subgradient method is guaranteed to converge to the global maximizer of $h(\alpha,\beta)$ at a sublinear rate; see~\cite{nesterov1983method,boyd2003subgradient,nesterov2009primal} for more details on the convergence rate of subgradient method.

We close this section by revisiting Example~\ref{ex:tree}, demonstrating that Algorithm~\ref{alg:Decomposition} can indeed achieve substantially improved optimality gaps.
\setcounter{example}{0}
\begin{example}[Continued]
 The Fenchel dual of \eqref{example1} is
\begin{subequations}\label{fenchelrelaxation}
\begin{align}
    \zeta_p=\max_{\alpha,\beta}&-0.4f^*(\alpha,\beta_1,\beta_2)+\min_{x,z}\bigg\{-1.3x_1+(-2.5+0.4\alpha)x_2+4.6x_3\\
    &-(7.8+0.4\alpha)x_4+3x_1^2+5.6x_2^2+3x_3^2+1.6x_4^2-1.5x_1x_2\\
    &-x_2x_3+2z_1+(2-0.4\beta_1)z_2+2z_3+(2-0.4\beta_2)z_4\bigg\}\\
    \text{s.t.}\;& x_i(1-z_i)=0,z_i\in\{0,1\}, \qquad i=1,\dots,4.
\end{align}
\end{subequations}
Table~\ref{tableFenchelexample1} shows the first nine iterations of Algorithm~\ref{alg:Decomposition}.
\begin{table}[htb]
\centering
\caption{Algorithm~\ref{alg:Decomposition} applied to problem \eqref{fenchelrelaxation} with $s_k=\frac{1}{1.01^k}$.}\label{tableFenchelexample1}
\begin{tabular}{c|c|c|c|c|c|c} \hline
Iteration & $\bar x$ & $\alpha$ & $\beta_{1}$ & $\beta_{2}$ & $\zeta_p$ & Gap\\ \hline
1 & (0,0,-1.53,6.50) & 0 & 0 & 0 & -24.87 & 67.93\%\\ 
2 & (0,0,-1.53,6.16) & -1.0 & 0.25 & 0 & -22.44 & 57.23\%\\ 
3 & (0,0,-1.53,5.83) & -1.99 & 0.25 & 0 & -20.36 & 46.04\%\\ 
4 & (0,0,-1.53,5.50) & -2.97 & 0.25 & 0 & -18.62 & 34.78\%\\ 
5 & (0,0,-1.53,5.18) & -3.94 & 0.25 & 0 & -17.21 & 24.02\%\\ 
6 & (0,0,-1.53,4.86) & -4.90 & 0.25 & 0 & -16.13 & 14.45\%\\ 
7 & (0,0,-1.53,4.54) & -5.85 & 0.25 & 0 & -15.36 & 6.84\%\\ 
8 & (0,0,-1.53,4.23) & -6.79 & 0.25 & 0 & -14.90 & 1.88\%\\ 
9 & (0,0,-1.53,3.92) & -7.72 & 0.25 & 0 & -14.73 & $<$0.01\%\\ \hline
\end{tabular} 
 
\end{table}
\end{example}

\section{Path Decomposition}\label{sec:decomposition}

In the previous section, we showed that if $Q$ does not possess a tridiagonal structure, then it is possible to relax its ``problematic'' elements via their Fenchel duals, and leverage Algorithm~\ref{alg:SP} to solve the resulting relaxation. In this section, our goal is to explain how to select the nonzero elements of $Q$ to be relaxed via our proposed method. In particular, our goal is to obtain the best permutation matrix $P$ such that $PQP^\top$ is close to tridiagonal.

To achieve this goal, we propose a path decomposition method over $\mathcal{G}$, where the problem of finding the best permutation matrix for $Q$ is reformulated as finding a maximum weight subgraph of $\mathcal{G}$, denoted as  $\widetilde{\mathcal{G}}$, that is a union of paths. In particular, define $y_{ij}$ as an indicator variable that takes the value 1 if and only if  edge $(i,j)$ is included in the subgraph. Therefore, the problem of finding $\widetilde{\mathcal{G}}$  reduces to:
\begin{subequations}\label{path_decomp}
\begin{align}
    p^*=\max &\sum_{(i,j) \in E}{|Q_{ij}|y_{ij}}\\
s.t.&\sum_{j \in \delta (i)}{y_{ij}} \leq 2, && \forall i = 1,2,\dots, n\label{eq_degree2}\\
&\sum_{i,j\in S}y_{ij}\leq |S|-1 && \forall S\subseteq \{1,2,\dots,n\}\label{eq_noclyce}\\
 &y_{ij} \in \{0,1\}, &&  \forall (i,j) \in E.
\end{align}
\end{subequations}
where $\delta(i)$ denotes the neighbors of node $i$ in $\mathcal{G}$. Let the objective function evaluated at a given $y$ be denoted as $p(y)$. Moreover, let $y^*$ and $p^*$  be an optimal solution and its corresponding objective value respectively.
Constraints~\eqref{eq_degree2} ensure that the constructed graph is the union of cycles and paths, whereas  constraints~\eqref{eq_noclyce} are cycle-breaking constraints~\cite{kruskal1956shortest,magnanti1995optimal,wolsey2020integer}. Despite the exponential number of constraints \eqref{eq_noclyce}, it is known that cycle elimination constraints can often be efficiently separated~\cite{wolsey2020integer}. Nonetheless, our next result shows that  problem~\eqref{path_decomp} is indeed NP-hard.
\begin{theorem}
Problem~\eqref{path_decomp} is NP-hard.
\end{theorem}
\begin{proof}
We use a reduction from Hamiltonian path problem: given an arbitrary (unweighted) graph $\mathcal{G}$, the Hamiltonian path problem asks whether there exists a simple path that traverses every node in $\mathcal{G}$. It is known that Hamiltonian path problem is NP-complete~\cite{garey1979computers}. 

Given an arbitrary graph $\mathcal{G}(N,E)$, construct an instance of~\eqref{path_decomp} with $|Q_{ij}| = 1$ if $(i,j)\in E$, and $|Q_{ij}|=0$ otherwise. 
Let us denote the optimal solution to the constructed problem as $y^*$, and the graph induced by this solution as ${\mathcal{G}^*}(N, E^*)$. In other words, $(i,j)\in E^*$ if and only if $y^*_{ij}=1$. It is easy to see that $\sum_{(i,j)\in E}y^*_{ij}\leq n-1$; otherwise, the graph ${\mathcal{G}}^*$ contains a cycle, which is a contradiction. Therefore, we have $p(y^*)\leq n-1$.  We show that, we have $p(y^*) = n-1$ if and only if $\mathcal{G}$ contains a Hamiltonian path. This immediately completes the proof.

First, suppose that $\mathcal{G}$ has a Hamiltonian path. Therefore, there exists a path in $\mathcal{G}$ with exactly $n-1$ edges. A solution $\tilde{y}$ defined as  $\tilde y_{ij}=1$ for every edge $(i,j)$ in the path, and $\tilde y_{ij}=0$ otherwise is feasible for the constructed instance of~\eqref{path_decomp}, and it has the objective value $p(\tilde y) = p(y^*) = n-1$. Conversely, suppose that $p(y^*) = n-1$. Then, the graph ${\mathcal{G}}^*$ has exactly $n-1$ edges, and it is a union of paths. It is easy to see that if ${\mathcal{G}}^*$ has at least two components, then $|E^*|<n-1$, which is a contradiction. Therefore, ${\mathcal{G}}^*$ is a Hamiltonian path.
\end{proof}

Due to hardness of~\eqref{path_decomp}, we propose in this section an approximation algorithm based on the following idea:
\begin{enumerate}
\item Find a vertex disjoint path/cycle cover of $\mathcal{G}$, that is, a subset $\widetilde{E}$ of the edges of $E$ such that, in the induced subgraph of $\mathcal{G}$, each connected component is either a cycle or a path.
\item From each cycle, remove the edge $(i,j)\in \widetilde{E}$ with least value $|Q_{ij}|$.
\end{enumerate}

Note that a vertex disjoint cycle cover can be found by solving a bipartite matching problem \cite{tutte1954short} on an auxiliary graph, after using a node splitting technique. Specifically, create graph $\mathcal{G}_M=(V_M,E_M)$ with $V_M=N\cup \left\{1',2',\dots,n'\right\}$ and $E_M$ that is determined as follows: if $(i,j)\in E$, then $(i,j')\in E_M$ and $(i',j)\in E_M$. Then any matching on $\mathcal{G}_M$ corresponds to a cycle cover in $\mathcal{G}$, with edge $(i,j')$ in the matching encoding that ``$j$ follows $i$ in a cycle." Figure~\ref{fig:matching} illustrates how to obtain cycle covers via bipartite matchings. 

\begin{figure}[h!]
\centering
\includegraphics[trim={6cm 0.5cm 11cm 0.5cm}, clip, width=0.3\textwidth]{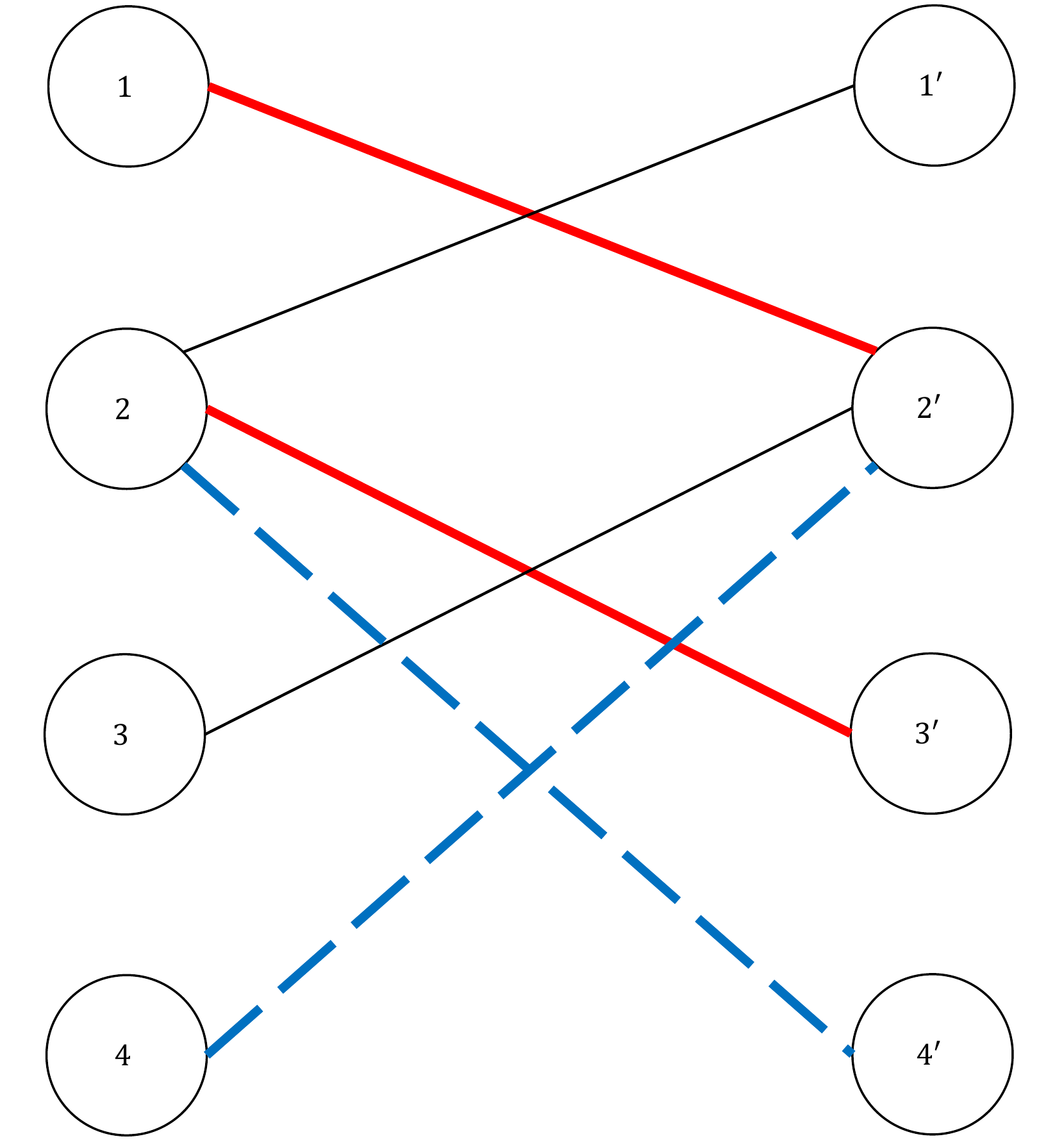}
\caption{Graph $\mathcal{G}_M$ corresponding to the support graph in Figure~\ref{fig:4pointgraph} (left). Bold red: Max. cardinality matching corresponding to the decomposition shown in Figure~\ref{fig:4pointgraph} (right). Dashed blue: Alternative max. cardinality matching corresponding to using edge $(2,4)$ twice, resulting in a length two cycle.}
\label{fig:matching}
\end{figure}

In the resulting decomposition from this method, each connected component is a cycle, possibly of length two (that is, using the same edge twice). However, in inference problems with graphical models, graph $\mathcal{G}$ is often bipartite (see Figure~\ref{fig: 2D}), in which case we propose an improved method which consists of solving the integer optimization problem
\begin{subequations}\label{cycle_decomp}
\begin{align}
    \max &\sum_{(i,j) \in E}{|Q_{ij}|y_{ij}}\\
s.t.&\sum_{j \in \delta (i)}{y_{ij}} \leq 2, && \forall i = 1,2,\dots, n\label{cycle_degree2}\\
 &y_{ij} \in \{0,1\}, &&  \forall (i,j) \in E.
\end{align}
\end{subequations}
Problem~\eqref{cycle_decomp} is obtained from~\eqref{path_decomp} after dropping the cycle elimination constraints~\eqref{eq_noclyce}. Note that in any feasible solution of \eqref{cycle_decomp} each edge can be used only once, thus preventing cycles of length two. Moreover, some of the connected components may already be paths.
Finally, we note that~\eqref{cycle_decomp} is much simpler than~\eqref{path_decomp}, since it can be solved  in polynomial time for certain  graph structures.

\begin{proposition}
For bipartite $\mathcal{G}$, the linear programming relaxation of the problem~\eqref{cycle_decomp} is exact.
\end{proposition}
\begin{proof}
It is easy to verify that the constraint matrix for~\eqref{cycle_decomp} is totally unimodular if $\mathcal{G}$ is bipartite, and therefore, the linear programming relaxation of~\eqref{cycle_decomp} has integer solutions. 
\end{proof}

Let the optimal solution to~\eqref{cycle_decomp} be denoted as $\hat y$. Define the weighted graph $\widehat{\mathcal{G}}(N, \widehat E)\subseteq \mathcal{G}(N, E)$ induced by $\hat y$ such that $(i,j)\in \widehat E$ with weight $|Q_{ij}|$ if and only if $\hat y_{ij} = 1$. Suppose that the graph $\widetilde{\mathcal{G}}(N, \widetilde E)$ is obtained after eliminating a single edge with smallest weight from every cycle of $\widehat{\mathcal{G}}$. Finally, define $\tilde y$ such that $\tilde y_{ij} = 1$ if $(i,j)\in \widetilde E$, and $\tilde y_{ij} = 0$ if $(i,j)\in E\backslash \widetilde{E}$. Next, we show that the above procedure leads to a $2/3$-approximation of~\eqref{path_decomp} in general, and a $3/4$-approximation for bipartite graphs.

\begin{theorem}
We have 
$\frac{p(\tilde{y})}{p^*}\leq\frac{2}{3}$. Moreover, if $\mathcal{G}$ is bipartite, then $\frac{p(\tilde{y})}{p^*}\leq\frac{3}{4}$.
\end{theorem}
\begin{proof}
Recall that $\widehat{\mathcal{G}}$ is a union of paths and cycles. Let $\gamma$ and $\eta$  denote the number of path and cycle components in $\widehat{\mathcal{G}}$, respectively. Moreover,  let $\mathcal{P} = \{P_1,P_2,\dots, P_\gamma\}$ and $\mathcal{C} = \{C_1,C_2,\dots, C_\eta\}$ denote the set of paths and cycle components in $\widehat{\mathcal{G}}$. Clearly, we have $p(\hat{y})\leq p^*\leq p(\hat{y})$, since problem~\eqref{cycle_decomp} is a relaxation of~\eqref{path_decomp}, and $\hat{y}$ is a feasible solution to~\eqref{path_decomp}. On the other hand, we have
\begin{align}
    p(\hat{y}) &= \sum_{P_k\in\mathcal{P}}\sum_{(i,j)\in P_k}|Q_{ij}| + \sum_{C_k\in\mathcal{C}}\sum_{(i,j)\in C_k}|Q_{ij}|\nonumber\\
    &\geq p(\tilde{y})\nonumber\\
    &= \sum_{P_k\in\mathcal{P}}\sum_{(i,j)\in P_k}|Q_{ij}| + \sum_{C_k\in\mathcal{C}}\left(\left(\sum_{(i,j)\in C_k}|Q_{ij}|\right)-\min_{(i,j)\in C_k}|Q_{ij}|\right)\nonumber\\
    &\geq \sum_{P_k\in\mathcal{P}}\sum_{(i,j)\in P_k}|Q_{ij}| + \frac{2}{3}\sum_{C_k\in\mathcal{C}}\sum_{(i,j)\in C_k}|Q_{ij}|\nonumber\\
    &\geq \frac{2}{3}\left(\sum_{P_k\in\mathcal{P}}\sum_{(i,j)\in P_k}|Q_{ij}| + \sum_{C_k\in\mathcal{C}}\sum_{(i,j)\in C_k}|Q_{ij}|\right)\nonumber\\
    &=\frac{2}{3}p(\hat{y}),\nonumber
\end{align}
where in the second inequality, we used the fact that removing an edge with the smallest weight from a cycle can reduce the weight of that cycle by at most a factor of $\frac{2}{3}$. This implies that $\frac{2}{3} p(\hat{y})\leq p(\hat{y})\leq p^*\leq p(\hat{y})$, and hence, $\frac{p(\tilde{y})}{p^*}\leq\frac{2}{3}$. The last part of the theorem follows since bipartite graphs do not contain cycles of length 3, thus each cycle of $\widehat{\mathcal{G}}$ has length four or more. Therefore, removing an edge with the smallest weight from a cycle of a bipartite graph reduces the weight by at most a factor of $\frac{3}{4}$.
\hfill $\hfill\square$
\end{proof}

\begin{remark}
It can be easily shown that the procedure applied to a pure cycle cover of $\mathcal{G}$, including cycles of length 2, would lead to a 1/2-approximation. Thus the proposed method indeed delivers in theory higher quality solutions for bipartite graphs, reducing the optimality gap of the worst-case performance by half.\hfill \hfill \qed
\end{remark}

\section{Computational Results}\label{sec:computations}

We now report illustrative computational experiments showcasing the performance of the proposed methods. First, in Section~\ref{subsec:compute tridiagonal}, we demonstrate the performance of Algorithm~\ref{alg:SP} on instances with tridiagonal matrices. Then, in Section~\ref{subsec:compute graphical}, we discuss the performance of Algorithm~\ref{alg:Decomposition} on instances inspired by inference with graphical models.

\subsection{Experiments with tridiagonal instances}\label{subsec:compute tridiagonal}
In this section, we consider  instances with tridiagonal matrices. We  compare the performance Algorithm~\ref{alg:SP}, the direct $\mathcal{O}(n^3)$ method mentioned in the beginning of Section~\ref{sec:efficient_alg}, and the big-M mixed-integer nonlinear optimization formulation \eqref{eq:map_miqo}, solved using Gurobi v9.0.2. All experiments are run on a Lenovo laptop with a 1.9GHz Intel$\circledR$Core$^{\text{TM}}$ i7-8650U CPU and 16 GB main memory; for Gurobi, we use a single thread and a time limit of one hour, and stop whenever the optimality gap is 1\% or less. 

In the first set of experiments, we construct tridiagonal matrices $Q\in\mathbb{R}^{N\times N}$ and vectors $a,c\in\mathbb{R}^N$ randomly as $c=\text{Uniform}[-10,3], a=\text{Uniform}[0,1], \linebreak Q_{i,i+1}=\text{Uniform}[-2,2], Q_{ii}=|Q_{i,i-1}|+|Q_{i,i+1}|+\text{Uniform}[0,4]$. Table~\ref{tab:tridiagonal differnet N} reports the time in seconds required to solve the instances by each method considered, as well as the gap and the number of branch-and-bound nodes reported by Gurobi, for different dimensions $n\leq 200$. Each row represents the average over 10 instances generated with the same parameters. 

\begin{table}[bhtp]
\centering
\caption{Perfomance solving tridiagonal instances. }
\begin{threeparttable}
\begin{tabular}{c|c|c|c|c | c}\hline
\textbf{Metric}  &\textbf{Method}  & $\bm{n=10}$    & $\bm{n=50}$    & $\bm{n=100}$   & $\bm{n=200}$  \\ \hline
&Algorithm~\ref{alg:SP} & $<$0.1  & $<$0.1  & $<$0.1  & $<$0.1\\ 
Time(s) &Direct $\mathcal{O}(n^3)$& $<$0.1  & 0.2   & 1.6   & 16.6 \\ 
&Big-M  & $<$0.1  & 0.5   & 43.4  & TL \\ \hline
B\&B nodes&\multirow{2}{*}{Big-M} & 14    & 2,764  & 404,475 & 17,965,177\\
Gap & & 0\%     & 0\%     & $<$1.0\% & 2.5\% \\ \hline
\end{tabular}%
\begin{tablenotes}
\footnotesize
\item TL: Time Limit (1 hour).
\end{tablenotes}
\end{threeparttable}
\label{tab:tridiagonal differnet N}%
\end{table}%
As expected, mixed-integer optimization approaches struggle in instances with $n=200$, whereas the polynomial time methods are much faster. Moreover, as expected, Algorithm~\ref{alg:SP}, with worst-case complexity of $\mathcal{O}(n^2)$, is substantially faster than the direct $\mathcal{O}(n^3)$ method. To better illustrate the scalability of the proposed methods, we report in Figure~\ref{fig:Tridiagonal Time} the time used by the polynomial time methods to solve instances with $10\leq n\leq 10,000$. We see that the direct $\mathcal{O}(n^3)$ method requires over 10 minutes to solve instances with $n=500$, and over one hour for instances with $n\geq 1,000$. In contrast, the faster Algorithm~\ref{alg:SP} can solve instances with $n\leq 1,000$ in under one second, and instances with $n=10,000$ in less than one minute. We also see that the practical performance of both methods matches the theoretical complexity.

\begin{figure}[htbp]
\centering
\includegraphics[width=0.8\textwidth]{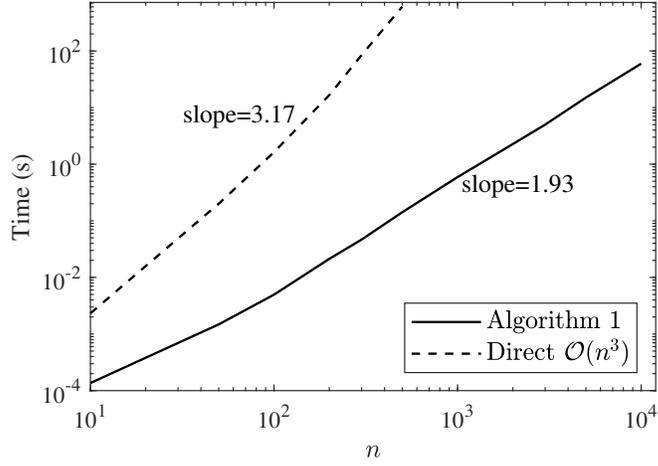}
\caption{Time (in seconds) required to solve tridiagonal problems with $10\leq n\leq 10,000$, in logarithmic scale.}
\label{fig:Tridiagonal Time}
\end{figure}

We also tested Algorithm~\ref{alg:SP} in inference problems with one-dimensional graphical models of the form \eqref{eq:time-series}, which are naturally tridiagonal. For this setting, we use the synthethic data $\{y_t\}_{t=1}^n\in \R^n$ with $n=1,000$ used in \cite{atamturk2021sparse}, available online at \url{https://sites.google.com/usc.edu/gomez/data}. Instances are classified according to a noise parameter $\sigma$, corresponding to the standard deviation of the noise $\epsilon_i$, see Section~\ref{sec:background} (all noise terms have the same variance). The results are reported in Table~\ref{tab:tridiagonal signal}.
\begin{table}[htbp]
\centering
\caption{Results with one-dimensional graphical models \eqref{eq:time-series}, $n=1,000$. We set $a_i=\mu$ for all $i\in N$, and choose $\mu$ so that approximately $\|x\|_0=0.1n$ in an optimal solution.}
\begin{tabular}{c c|c | c c c }\hline
\multirow{2}{*}{$\bm{\sigma}$} &  \multirow{2}{*}{$\bm{\mu}$} & \textbf{Algorithm~\ref{alg:SP}} & \multicolumn{3}{c}{\textbf{Big-M}} \\
& & \textbf{Time(s)}  & \textbf{Time(s)}  & \textbf{Gap}   & \textbf{Nodes} \\\hline
0.10   & 0.01  & 0.8  & TL  & 1.8\% & 4,034,520 \\
0.50   & 0.02  & 0.8  & TL  & 40.6\% & 4,474,262
\\
1.00   & 0.12  & 0.8  & TL  & 42.2\% & 3,749,981 \\\hline
\end{tabular}%
\label{tab:tridiagonal signal}%
\end{table}%

 Once again, Algorithm~\ref{alg:SP} is substantially faster than the big-M formulation solved using Gurobi. More interestingly perhaps are how the results reported here compare with those of \cite{atamturk2021sparse}. In that paper, the authors propose a conic quadratic relaxation of problem\footnote{They consider a slightly different term, where the sparsity is imposed via a cardinality constraint $a^\top z\leq k$ instead of a penalization in the objective.} \eqref{eq:time-series}, and solve this relaxation using the off-the-shelf solver Mosek. The authors report that solving this relaxation requires two seconds in these instances. Note that solution times are not directly comparable due to using different computing environments. Nonetheless, we see that, using Algorithm~\ref{alg:SP}, the mixed-integer optimization problem \eqref{eq:time-series} can be solved \emph{to optimality} in approximately the same time required to solve the (not necessarily tight) convex relaxation proposed in \cite{atamturk2021sparse}. Moreover, Algorithm~\ref{alg:SP} can be used with arbitrary tridiagonal matrices $Q\succeq 0$, whereas the method of \cite{atamturk2021sparse} requires the additional assumption that $Q$ is a Stieltjes matrix.

\subsection{Inference with two-dimensional graphical models}\label{subsec:compute graphical}
In the previous section, we reported experiments with tridiagonal matrices, where Algorithm~\ref{alg:SP} delivers the optimal solution of the mixed-integer problem.
In this section, we report our computational experiments with solving inference problems \eqref{eq:miqo} using Algorithm~\ref{alg:Decomposition} (which is not guaranteed to find an optimal solution) and the big-M formulation. In the considered instances, graph $\mathcal{G}$ is given by the two-dimensional lattice depicted in Figure~\ref{fig: 2D}, that is, elements of $N$ are arranged in a grid and there are edges between horizontally/vertically adjacent vertices. We consider instances with grid sizes $10\times 10$ and $40\times 40$, thus resulting in instances with $n=100$ and $n=1,600$, respectively. The data for $y$ are generated similarly to \cite{he2021comparing}, where $\sigma$ is the standard deviation of the noise terms $\epsilon_i$. The data is available online at \url{https://sites.google.com/usc.edu/gomez/data}. 

We test two different step sizes\footnote{For step size $s_k=1/k$, we modify line~\ref{line:decomposition_dual} of Algorithm~\ref{alg:Decomposition} to $(\alpha,\beta)\leftarrow (\alpha,\beta)+s_k \rho(\bar x,\bar z)$ (without normalization), since this version performed better in our computations.} $s_k=1/k$ and $s_k=(1.01)^{-k}$ for Algorithm~\ref{alg:Decomposition}. For both the big-M formulation and Algorithm~\ref{alg:Decomposition}, we stop whenever the proven optimality gap  is less than 1\%. Moreover,  for the big-M formulation, we set a time limit of one hour, and  for Algorithm~\ref{alg:Decomposition} we set an iteration limit of $\bar k=300$ in instances with $n=100$, and $\bar k=100$ in instances with $n=1,600$. 
Tables~\ref{tab:10graphical} and \ref{tab:40graphical} report results with $n=100$ and $n=1,600$, respectively. They show the time (in seconds) and gaps proven by each method, as well as the number of iterations for Algorithm~\ref{alg:Decomposition} and the number of branch-and-bound nodes explored by Gurobi. Each row represents an average over five instances.

\begin{table}[htbp]
\setlength{\tabcolsep}{1pt}
\centering
\caption{$10\times 10$ Graphical model, i.e., $n=100$. We set $a_i=\mu$ for all $i\in N$, and choose $\mu$ so that in an optimal solution, $\|x\|_0$ approximately matches the number of nonzeros of the underlying signal.}
\begin{tabular}{cc|ccc|ccc|ccc}\hline
\multirow{2}[0]{*}{$\bm{\sigma}$} & \multirow{2}[0]{*}{$\bm{\mu}$}  & \multicolumn{3}{|c|}{\textbf{Algorithm~\ref{alg:Decomposition},} \bm{$s_k=(1.01)^{-k}$}} & \multicolumn{3}{|c|}{\textbf{Algorithm~\ref{alg:Decomposition},} $\bm{s_k=1/k}$} & \multicolumn{3}{|c}{\textbf{Big-M}} \\
&     &  \textbf{Iter.} & \textbf{Time(s)} & \textbf{Gap}   & \textbf{Iter.} & \textbf{Time(s)} & \textbf{Gap}   & \textbf{Nodes}  & \textbf{Time(s)} & \textbf{Gap} \\\hline
0.02  & 0.5   & 14    & 0.2   & $<$1\%  & 9     & 0.2   & $<$1\%  & 68    & 0.1   & 0.0\% \\
0.1   & 0.5   & 22    & 0.4   & $<$1\%  & 7     & 0.1   &$<$1\%  & 556   & 0.3   & 0.0\% \\
0.3   & 0.1   & 52    & 0.9   & $<$1\%  & 9     & 0.2   & $<$1\%  & 1,431,538 & 268.6 & $<$1\% \\
0.5   & 0.1   & 193   & 3.2   & $<$1\%  & 27    & 0.5   & $<$1\%  & 14,116,955 & TL & 4.7\% \\\hline
\end{tabular}%
\label{tab:10graphical}%
\end{table}%

\begin{table}[htbp]
\setlength{\tabcolsep}{1pt}
\centering
\caption{$40\times 40$ Graphical model, i.e., $n=1600$.We set $a_i=\mu$ for all $i\in N$, and choose $\mu$ so that in an optimal solution, $\|x\|_0$ approximately matches the number of nonzeros of the underlying signal.}
\begin{tabular}{cc|ccc|ccc|ccc}\hline
\multirow{2}[0]{*}{$\bm{\sigma}$} & \multirow{2}[0]{*}{$\bm{\mu}$}  & \multicolumn{3}{|c|}{\textbf{Algorithm~\ref{alg:Decomposition},} \bm{$s_k=(1.01)^{-k}$}} & \multicolumn{3}{|c|}{\textbf{Algorithm~\ref{alg:Decomposition},} $\bm{s_k=1/k}$} & \multicolumn{3}{|c}{\textbf{Big-M}} \\
&     &  \textbf{Iter.} & \textbf{Time(s)} & \textbf{Gap} & \textbf{Iter.} & \textbf{Time(s)} & \textbf{Gap}   & \textbf{Nodes}  & \textbf{Time(s)} & \textbf{Gap} \\\hline
0.02  & 0.05  & 9     & 28.3  & $<$1\%  & 15    & 47.1  & $<$1\%  & 5,485  & 23.1  & $<$1\% \\
0.1   & 0.05  & 8     & 27.2  & $<$1\%  & 13    & 43.6  & $<$1\%  & 618,310 & TL & 3.9\% \\
0.3   & 0.025 & 73    & 224.2 & $<$1\%  & 20    & 62.4  & $<$1\%  & 639,431 & TL & 21.4\% \\
0.5   & 0.05  & 100   & 303.1 & 1.6\% & 61    & 190.5 & 1.0\% & 669,872 & TL & 30.9\% \\
\hline
\end{tabular}%
\label{tab:40graphical}%
\end{table}%

We see that the big-M formulation can be solved fast for low noise values, but struggles in high-noise regimes. For example, if $n=100$, problems with $\sigma\leq 0.1$ are solved in under one second, while problems with $\sigma=0.5$ cannot be solved within the time limit. For instances with $n=1,600$, gaps can be as large as 30\% in high noise regimes. In contrast, Algorithm~\ref{alg:Decomposition} consistently delivers solutions with low optimality gaps. For example, the worst reported gap is 1.6\% ($n=1,600$, $\sigma=0.5$, $s_k=1.01^{-k}$), but is often much less. In particular, using step size $s_k=1/k$, average optimality gaps of at most $1.0\%$ are obtained in all cases. In terms of run times, Algorithm~\ref{alg:Decomposition} runs in seconds on instances with $n=100$, and in under five minutes on instances with $n=1,600$. 

In summary, for the instances that are not solved to optimality using the big-M formulation, Algorithm~\ref{alg:Decomposition} is able to reduce the optimality gaps by at least an order of magnitude while requiring only a small fraction of the computational time.

\bibliographystyle{abbrv}
\bibliography{reference}

\end{document}